\documentclass[a4paper,10pt]{article}

\usepackage[normalem]{ulem}

\usepackage[latin1]{inputenc}
\usepackage[english]{babel}
\usepackage{amsmath}
\usepackage{amsthm}
\usepackage{amssymb} 
\usepackage{amsfonts}
\usepackage[dvips]{graphics,graphicx}
\usepackage{subfigure}

\DeclareSymbolFont{AMSb}{U}{msb}{m}{n}
\DeclareMathSymbol{\N}{\mathbin}{AMSb}{"4E}
\DeclareMathSymbol{\Z}{\mathbin}{AMSb}{"5A}
\DeclareMathSymbol{\R}{\mathbin}{AMSb}{"52}
\DeclareMathSymbol{\Q}{\mathbin}{AMSb}{"51}
\DeclareMathSymbol{\C}{\mathbin}{AMSb}{"43}

\newcommand{\elltwo}{L^2_{[-1,1]}}

\newcommand{\DX}[1]{\, {\rm d}#1}

\newcommand{\ii}{\textrm{i}}
\newcommand{\BIGO}{{\mathcal O}}

\newcommand{\CPT}{\ensuremath{\cos \frac{\pi}{T}}}
\newcommand{\CPTX}{\ensuremath{\cos \frac{\pi}{T} x}}
\newcommand{\SPTX}{\ensuremath{\sin \frac{\pi}{T} x}}
\newcommand{\CKPTX}{\ensuremath{\cos \frac{\pi}{T} k x}}
\newcommand{\SKPTX}{\ensuremath{\sin \frac{\pi}{T} k x}}
\newcommand{\CHW}{\frac{1}{\sqrt{1-y^2}}} 
\newcommand{\CHWS}{\sqrt{1-y^2}} 

\newtheorem{lemma}{Lemma}[section]
\newtheorem{theorem}[lemma]{Theorem}

\newtheorem{problem}[lemma]{Problem}

\usepackage{color}

\title{On the resolution power of Fourier extensions \\for oscillatory functions}
\author{Ben Adcock\footnote{Department of Mathematics, Simon Fraser University, Canada (ben\_adcock@sfu.ca)}~ and Daan Huybrechs\footnote{Department of Computer Science, K.U.Leuven, Belgium (daan.huybrechs@cs.kuleuven.be)}}
\date{}

\begin{document}

\maketitle

\begin{abstract}
Functions that are smooth but non-periodic on a certain interval possess Fourier series that lack uniform convergence and suffer from the Gibbs phenomenon. However, they can be represented accurately by a Fourier series that is periodic on a larger interval. This is commonly called a Fourier extension.  When constructed in a particular manner, Fourier extensions share many of the same features of a standard Fourier series.  In particular, one can compute Fourier extensions which converge spectrally fast whenever the function is smooth, and exponentially fast if the function is analytic, much the same as the Fourier series of a smooth/analytic and periodic function.

With this in mind, the purpose of this paper is to describe, analyze and explain the observation that Fourier extensions, much like classical Fourier series, also have excellent resolution properties for representing oscillatory functions. The \emph{resolution power}, or required number of degrees of freedom per wavelength, depends on a user-controlled parameter and, as we show, it varies between $2$ and $\pi$. The former value is optimal and is achieved by classical Fourier series for periodic functions, for example. The latter value is the resolution power of algebraic polynomial approximations.  Thus, Fourier extensions with an appropriate choice of parameter are eminently suitable for problems with moderate to high degrees of oscillation.
\end{abstract}

\section{Introduction}\label{s:introduction}
In many physical problems, one encounters the phenomenon of oscillation.  When approximating the solution to such a problem with a given numerical method, this naturally leads to the question of resolution power.  That is, how many degrees of freedom are required in a given scheme to resolve such oscillations?  Whilst it may be impossible to answer this question in general, important heuristic information about a given approximation scheme can be gained by restricting ones interest to certain simple classes of oscillatory functions (e.g. complex exponentials for problems in bounded intervals).  

Resolution power represents an \textit{a priori} measure of the efficiency of a numerical scheme for a particular class of problems.  Schemes with low resolution power require more degrees of freedom, and hence increased computational cost, before the onset of convergence.  Conversely, schemes with high resolution power capture oscillations with fewer degrees of freedom, resulting in decreased computational expense.

Consider the case of the unit interval $[-1,1]$ (the primary subject of this paper).  Here one typically studies the question of resolution via the complex exponentials
\begin{equation} \label{E:fmodel}
f(x) = \exp (i \omega \pi x),\quad \omega \in \R.
\end{equation}
To this end, let $\phi_{n}(f)$, $n=1,2,\ldots$ be a sequence of approximations of the function $f(x) = \exp (i \pi \omega x)$ which converges to $f$ as $n \rightarrow \infty$ (here $n$ is the number of degrees of freedom in the approximation $\phi_n(f)$).  For $0<\epsilon < 1$, let $n(\epsilon,\omega)$ be the minimal $n$ such that
\begin{equation*}
\| f - \phi_n(f) \|_{L^2_{[-1,1]}} < \epsilon.
\end{equation*}
We now define the \emph{resolution constant} $r$ of the approximation scheme $\{ \phi_n \}$ as
\begin{equation*}
r = \limsup_{\epsilon \rightarrow 1^{-}} \lim_{\omega \rightarrow \infty} \frac{n(\epsilon,\omega)}{\omega}.
\end{equation*}
Note that $r$ need not be well defined for an arbitrary scheme $\{ \phi_n \}$ (for example, if $n(\epsilon,\omega)$ were to scale superlinearly in $\omega$).  However, for all schemes encountered in this paper, this will be the case. 

Loosely speaking, the resolution constant $r$ corresponds to the required number of degrees of freedom per wavelength to capture oscillatory behaviour; a common concept in the literature on oscillatory problems \cite{boyd,naspec}. In particular, we say that a given scheme has high (respectively low) resolution power if it has small (large) resolution constant.  It is also worth noting that, in many schemes of interest, the approximation $\phi_n(f)$ is based on a collocation at a particular set of $n$ nodes in $[-1,1]$.  In this circumstance, the resolution constant $r$ is equivalent to the number of \emph{points per wavelength} required to resolve an oscillatory wave (for further details, see \S \ref{ss:relation}).

With little doubt, the approximation of a smooth, periodic function via its truncated Fourier series is one of the most effective numerical methods known.  Fourier series, when computed via the FFT, lead to highly efficient, stable methods for the numerical solution of a large range of problems (in particular, PDE's with periodic boundary conditions).  A simple argument leads to a resolution constant of $r=2$ in this case (for periodic oscillations), with exponential convergence occurring once the number of Fourier coefficients exceeds $2\omega$.

However, the situation is altered completely once periodicity is lost.  The slow pointwise convergence of the Fourier series of a nonperiodic function, as well as the presence of $\BIGO(1)$ Gibbs oscillations near the domain boundaries, means that nonperiodic oscillations cannot be resolved by such an approximation.  A standard alternative is to approximate with a sequence of orthogonal polynomials (e.g. Chebyshev polynomials).  Such approximations possess exponential convergence, without periodicity, yet the resulting resolution constant increases to the value $r=\pi$, making such an approach clearly less than ideal.

\subsection{Fourier extensions}
The purpose of this paper is to discuss an alternative to polynomial approximation for nonperiodic functions; the so-called \textit{Fourier extension}.  Our main result is to show that Fourier extensions have excellent resolution properties.  In particular, there is a user-determined parameter $T \in (1,\infty)$ that allows for continuous variation of the resolution constant from the value $2$ (in the limit $T \rightarrow 1$), the figure corresponding to Fourier series, to $\pi$ ($T \rightarrow \infty$), the value obtained by polynomial approximations.  Thus, Fourier extensions are highly suitable tools for problems with oscillations at moderate to high frequencies.

Let us now describe Fourier extensions in more detail.  As discussed, Fourier series are eminently suitable for approximating smooth and periodic oscillatory functions.  It follows that a potential means to recover a highly accurate approximation of a nonperiodic function $f : [-1,1] \rightarrow \R$ is to seek to extend $f$ to a periodic function $g$ defined on a larger domain $[-T,T]$ and compute the Fourier series of $g$.  Unfortunately, unless $f$ is itself periodic, no periodic extension $g$ will be analytic, and hence only spectral convergence can be expected.  Preferably, for analytic $f$, we seek an approximation that converges exponentially fast.

To remedy this situation, rather than computing an extension $g$ explicitly, it was proposed in \cite{BoydFourCont,brunoFEP} to directly compute a Fourier representation of $f$ on the domain $[-T,T]$ via the so-called  \textit{Fourier extension problem}:

\begin{problem}\label{p:FEP}
Let $G_n$ be the space of $2T$-periodic functions of the form
\begin{equation}\label{E:g_form}
 g \in G_n: g(x) = \frac{\alpha_0}{2} + \sum_{k=1}^n \alpha_k \CKPTX + \beta_k \SKPTX.
\end{equation}
The \emph{(continuous) Fourier extension} of $f$ to the interval $[-T,T]$ is the solution to the optimization problem
\begin{equation}\label{E:ls}
g_n :=  \underset{g \in G_n}{\operatorname{argmin}}
\Vert f-g \Vert_{\elltwo}.
\end{equation}
\end{problem}

Note that there are infinitely many ways in which to define a smooth and periodic extension of a function $f : [-1,1] \rightarrow \mathbb{C}$, of which~\eqref{E:ls} is but one.  We say an extension $g$ is a \emph{Fourier} extension if $g \in G_n$ is a finite Fourier series on $[-T,T]$.  Note, however, that even within this stipulation, there are still infinitely many ways to define $g$.  We refer to the extension $g_n$ defined by~\eqref{E:ls} as the \textit{continuous} Fourier extension of $f$ (in the sense that it minimizes the continuous $L^2$-norm). 

As numerically observed in \cite{BoydFourCont,brunoFEP}, the convergence of $g_{n}$ to $f$ is exponential, provided $f$ is analytic.  When $T=2$, this was confirmed by the analysis presented in \cite{huybrechs2010fourier}. A pivotal role is played by the map
\begin{equation}\label{E:map_T2}
y = \cos \frac{\pi}{2} x.
\end{equation}
The importance of this map is that it transforms the trigonometric basis functions that span the space $G_n$ into polynomials in $y$. In this setting, Problem~\ref{p:FEP} reduces simply to the computation of expansions in a suitable basis of nonclassical orthogonal polynomials, since the least squares criterion corresponds exactly to an orthogonal projection in a particular weighted norm. Well-known results on orthogonal polynomials can then be used to establish convergence properties of the Fourier extension. We recall this analysis in greater detail in \S\ref{s:convergence}, along with its generalization to arbitrary $T$.

The map~\eqref{E:map_T2} demonstrates the close relationship that exists between Fourier series and polynomials. Note the similarity to the classical Chebyshev map $x = \cos \theta$ that takes Chebyshev polynomials to trigonometric basis functions. Compared to Chebyshev expansions, however, it is important to note that the roles of polynomials and trigonometric functions are interchanged. The Chebyshev expansion of a given function $f$ is a polynomial approximation, which is equivalent to the Fourier series of a related function. The Fourier extension of a function on the other hand is a trigonometric approximation, which is equivalent to the polynomial expansion of a related function. In that sense, Fourier extensions and Chebyshev expansions are dual to each other.

\subsection{Key results} \label{ss:results}
The main result of this paper is that the resolution constant $r = r(T)$ of the  continuous Fourier extension satisfies
\[
 r(T) \leq 2 T \sin \left ( \tfrac{\pi}{2T} \right ),\quad T \in (1,\infty).
\]
Accordingly, the Fourier extension $g_n$ of the function~\eqref{E:fmodel} will begin to converge once $n$ exceeds $\frac{1}{2}r(T) \omega $ (recall that $g_n$ involves $2n+1$ degrees of freedom).  In particular, we find that $r(T) \sim 2T$ for $T \approx 1$ and $r(T) \sim \pi$ as $T \rightarrow \infty$. Note that the resolution of the function $f(x) = \exp(i \omega \pi x)$ using classical Fourier series on $[-T,T]$ would require a minimum of $2T \omega$ degrees of freedom (whenever $f$ is periodic).  Thus, Fourier extensions exhibit comparable performance for $T$ close to $1$. However, as $T$ increases, $f$ can be resolved more efficiently via its Fourier extension than if it had been directly expanded in a Fourier series on $[-T,T]$.  In particular, the resolution power is bounded above for all $T$ by $\pi$, which is precisely the resolution constant for polynomial approximations.

Aside from establishing when asymptotic convergence will occur, we also show that the convergence in this regime is exponential, and that for all sufficiently large $n$ the rate is precisely $\rho = E(T)$, where
\[
E(T) = \cot^2 \left ( \tfrac{\pi}{4 T} \right ),
\]
(this result actually holds for all sufficiently analytic functions $f$, not just~\eqref{E:fmodel}). Here, we note that $E(1)=1$, implying no exponential convergence for $T=1$. This is of course a consequence of the Gibbs phenomenon of Fourier series on $[-1,1]$. Convergence is exponential for all $T$ greater than $1$, however, and the rate increases for larger $T$.

In summary, the main conclusion we draw in this paper is the following. Smaller $T$ yields better resolution power of the continuous  Fourier extension, at a cost of slower, but still exponential, convergence. Conversely, larger $T$ yields faster exponential convergence, at the expense of reduced resolution power.  Formally, one may also obtain a resolution constant of $2$ in the limit $n \rightarrow \infty$ by allowing $T \rightarrow 1$ as $n \rightarrow \infty$, and towards the end of the paper we shall discuss different strategies for doing this, some of which are quite effective in practice.

Unfortunately, the linear system of equations to be solved when computing~\eqref{E:ls} is severely ill-conditioned.  As a result, the best attainable accuracy with a continuous Fourier extension is of order $\sqrt{\epsilon}$, where $\epsilon$ is the machine precision used.  To overcome this, we present a new Fourier extension (referred to as the \textit{discrete} Fourier extension), based on a judicious choice of interpolation nodes, which leads to far less severe condition numbers.  Numerical examples demonstrate much higher attainable accuracies with this approach, typically of order $\epsilon$, whilst retaining both the same rates of convergence and resolution constant.

Inherent to Fourier extensions is the concept of redundancy: namely, there are infinitely many possible extensions of a given function.  As we explain, this not only leads to the ill-conditioning mentioned above, it also means that the result of a numerical computation may differ somewhat from the `theoretical' (i.e. infinite precision) continuous or discrete Fourier extension.   Later in the paper, we discuss this observation and its consequences in some detail.

\subsection{Relation to previous work}\label{ss:relation}
The question of resolution power of Fourier series and Chebyshev polynomial expansions was first developed rigorously by Gottlieb and Orszag \cite[p.35]{naspec} (see also \cite{OrszagIsraeli74}), who introduced and popularized the concept of points per wavelength.  Therein the figures of $2$ and $\pi$ respectively were derived, the latter being generalized to arbitrary Gegenbauer polynomial expansions in \cite{GottGibbs2} (in \S\ref{s:OPSresolution} we provide an alternative proof, valid for almost any orthogonal polynomial system (OPS)).  Resolution was also discussed in detail in \cite[chpt. 2]{boyd}, where, aside from Fourier and Chebyshev expansions, the extremely poor performance of finite difference schemes was noted.

One-dimensional Fourier extensions were, arguably, first studied in detail in \cite{BoydFourCont,brunoFEP}, where they were employed to overcome the Gibbs phenomenon in standard Fourier expansions, as well as the Runge phenomenon in equispaced polynomial interpolation (see also \cite{BoydRunge}).  Application to surface parametrizations was also explored in \cite{brunoFEP}.  

Similar ideas (typically referred to as \textit{Fourier embeddings}), were previously proposed to solve PDE's in complex geometries.  These work by embedding the domain in a larger bounding box, and computing a Fourier series approximation on this domain.  See \cite{BadeaDaripa03b,BertolottiHerbertSpalart92,boyd2005fourier,BuenoOrovio06,BuenoOrovioPerezGarcia06,BuenoOrovioPerezGarciaFenton06,BuffatLePenven11,Garbey00,JuangHong01,Lui09,pasquettiFourEmbed,SabetghadamSharafatmandjoorNorouzi09}.  Such methods can be shown to work well, in principle even for arbitrary regions, but they are typically prohibitively computationally expensive.

More recently, a very effective method was developed in \cite{bruno2010high,lyon2010high} to solve time-dependent PDE's in complex geometries.  This approach is based on a technique for obtaining one-dimensional Fourier extensions, known as the FE--Gram method, which is then combined with an alternating direction technique, as well as standard FFTs, to solve the PDE efficiently and accurately.   This approach has also been applied to Navier--Stokes equations \cite{albin2011}.  Interestingly, it is shown and emphasized in~\cite{lyon2010high} (see also \cite{albin2011}) that this method leads to an absence of dispersion errors (or pollution errors) -- another beneficial property for wave simulations shared with classical Fourier series, and very much related to resolution power.  

Having said this, we note at this point that the FE--Gram approach is quite different to the Fourier extensions we consider in this paper.  The FE--Gram approximation is not exponentially convergent, and in practice, the order of convergence is usually limited by the user to ensure stability in the overall PDE solver.  However, the FE--Gram approximation is designed specifically to be computed efficiently, in $\mathcal{O}(n \log n)$ time, using only function evaluations on an equispaced grid.  Although there has been some recent progress in the rapid computation of exponentially-convergent Fourier extensions of the form~\eqref{E:g_form} from equispaced data \cite{LyonFESVD,LyonFast}, this is not an issue we shall dwell on in this paper.  Thus, the main contribution of this paper is approximation-theoretic: we show that one can construct exponentially-convergent Fourier extensions with a resolution constant arbitrarily close to optimal.  In \S \ref{s:conclusions}
 we discuss ongoing and future work pertaining to these issues.

This aside, our use of Fourier extensions in this paper to resolve oscillatory functions is superficially quite similar to the Kosloff--Tal--Ezer mapping (see \cite{KTEmapped}, \cite[chpt 16.9]{boyd} and references therein, and more recently, \cite{haletrefethenconformal}) for improving the severe time-step restriction inherent in Chebyshev spectral methods.  Such an approach also improves the very much related property of resolution power.  In this technique, one replaces the standard Chebyshev interpolation nodes with a sequence of mapped points, and expands in a nonpolynomial basis defined via the particular mapping.  Roughly speaking, with Fourier extensions, the situation is reversed.  Rather than fixing interpolation nodes, one specifies a particular basis (i.e. the Fourier basis for a particular $T$) that gives rapidly convergent expansions and good resolution properties, and chooses an appropriate means for computing the extension (for example, a particular configuration of collocation nodes) via the 
corresponding mapping.

Finally, we remark in passing that there are a number of commonly used alternatives to Fourier extensions. Such methods typically arise from the desire to reconstruct a function directly from its Fourier coefficients (or pointwise values on an equispaced grid), whether via re-expanding in a sequence Gegenbauer polynomials \cite{GottGibbsRev,GottGibbs1}, or by smoothing the function by implicitly matching its derivatives at the domain boundary \cite{Eckhoff3}.  Whilst the latter retains a resolution constant of $2$, it only yields algebraic convergence of a finite order, and suffers from severe ill-conditioning.  Conversely, the Gegenbauer reconstruction procedure offers exponential convergence, but with a significant deterioration in resolution power \cite{GottGibbs2}.

The outline of the remainder of this paper is as follows.  In \S\ref{s:convergence} we detail the convergence of Fourier extensions for arbitrary $T>1$, and in \S \ref{ss:convergence_numerical} we address numerical computation.  In \S\ref{s:OPSresolution} we consider the resolution power of polynomial expansions, and in \S\ref{s:resolution} we derive the resolution constant for Fourier extensions.

\section{Fourier extensions on arbitrary intervals}\label{s:convergence}

In this section, we recall and generalize the analysis given in \cite{huybrechs2010fourier} for the case $T=2$. We first show a new result, namely that the continuous Fourier extension converges spectrally for all smooth functions $f$ and for all $T>1$ fixed.  Next, a more involved analysis in \S\ref{ss:convergence_exact} demonstrates that the continuous Fourier extension does in fact converge exponentially whenever $f$ is analytic. In doing so, we repeat some of the reasoning in \cite{huybrechs2010fourier} for the clarity of presentation, as well as for establishing notation that is needed later in the paper.

Before presenting these results, a word about terminology in order to avoid confusion. Throughout this paper, we will refer to the following three types of convergence of an approximation $f_n$ to a given function $f$.  We say that $f_n$ converges \textit{algebraically fast} to $f$ at \textit{rate $k$} if $\| f - f_n \| = \BIGO(n^{-k})$ as $n \rightarrow \infty$.  Conversely, $f_n$ converges \textit{spectrally fast} to $f$ if the error $\| f - f_n\|$ decays faster than any algebraic power of $n^{-1}$, and \textit{exponentially fast} if there exists some constant $\rho>1$ such that $\| f - f_n \| = \BIGO(\rho^{-n})$ for all large $n$.

\subsection{Spectral convergence}\label{ss:spectral}
Standard  approximations based on orthogonal polynomials (or Fourier series in the periodic case) converge exponentially fast provided $f$ is analytic, and spectrally fast if $f$ is only smooth \cite[chpts 2,5]{SMSD}.  Whenever $f$ has only finite regularity, convergence is algebraic at a rate determined by the degree of smoothness: specifically, if $f$ is $(k-1)$-times continuously differentiable in $[-1,1]$ and $f^{(k)}$ exists almost everywhere and is square-integrable (equivalently, $f \in H^{k}[-1,1]$ -- the $k$th standard Sobolev space of functions defined on $[-1,1]$), then $f_n$ converges algebraically fast at rate $k$ .  Our first result regarding Fourier extensions illustrates identical convergence in this setting:

\begin{theorem} \label{t:specconv}
Suppose that $f \in H^{k}[-1,1]$ for some $k \in \N$ and that $T_0 > 1$.  Then, for all $n \in \N$ and $T \geq T_0$,
\begin{equation}
\label{specineq}
\| f - g_n \|_{L^2_{[-1,1]}} \leq c_{k}(T_0) \left ( \frac{n \pi}{T} \right )^{-k} \| f \|_{H^{k}_{[-1,1]}},
\end{equation}
where $c_{k}(T_0) > 0 $ is independent of $n$, $f$ and $T$, $\| \cdot \|_{H^k_{[-1,1]}}$ is the standard norm on $H^{k}[-1,1]$ and $g_n$ is the continuous Fourier extension of $f$ on $[-T,T]$ defined by~\eqref{E:ls}.
\end{theorem}
\begin{proof}
Recall that there exists an extension operator $\mathcal{E} : H^{k}[-1,1] \rightarrow H^{k}(\R)$ with $\mathcal{E} f |_{[-1,1]} = f$ and $\| \mathcal{E} f \|_{H^k_{\R}} \leq c \| f \|_{H^k_{[-1,1]}}$ for some positive constant $c$ independent of $f$ \cite{adams}.

Let $\chi \in C^{\infty}(\R)$ be monotonically decreasing and satisfy $\chi(x) = 0$ for $x>T_0-1$ and $\chi(x) = 1$ for $x<0$.  We define the bump function $\mathcal{B} \in C^{\infty}(\R)$ by
\begin{equation*}
\mathcal{B}(x) = \left\{\begin{array}{cc}\chi(-x-1) & -T_0\leq x < -1 \\1 & -1 \leq x  \leq 1 \\ \chi(x-1) & 1 < x \leq T_0 \\ 0 & |x| > T_0. \end{array}\right.
\end{equation*}
Note that the function $\mathcal{B}(x) \mathcal{E} f(x) \in H^{k}(\R)$ and has support in $[-T_0,T_0]$.  Thus its restriction to $[-T,T]$ is periodic for any $T\geq T_0$.  Since $g_{n}$ minimizes the $L^2$ norm error over all functions from the set $G_n$, we have
\begin{equation*}
\| f - g_n \|_{L^2_{[-1,1]}} \leq \| f - (\mathcal{B} \mathcal{E} f)_n \|_{L^2_{[-1,1]}} \leq \| \mathcal{B} \mathcal{E} f - (\mathcal{B} \mathcal{E} f)_n \|_{L^2_{[-T,T]}},
\end{equation*}
where $(\mathcal{B} \mathcal{E} f)_n$ denotes the $n$th partial Fourier sum of $\mathcal{B}(x) \mathcal{E} f(x)$ on $[-T,T]$.  Since $\mathcal{B}(x) \mathcal{E} f(x) \in H^{k}[-T,T]$ and is periodic on $[-T,T]$, a well-known estimate \cite[eqn (5.1.10)]{SMSD} gives
\begin{equation*}
\| \mathcal{B} \mathcal{E} f - (\mathcal{B} \mathcal{E} f)_n \|_{L^2_{[-T,T]}} \leq \left ( \frac{n \pi}{T} \right )^{-k} \| (\mathcal{B} \mathcal{E} f)^{(k)} \|_{L^2_{[-T,T]}}.
\end{equation*}
Moreover, $\| (\mathcal{B} \mathcal{E} f)^{(k)} \|_{L^2_{[-T,T]}} = \| (\mathcal{B} \mathcal{E} f)^{(k)} \|_{L^2_{[-T_0,T_0]}} \leq c_{k}(T_0) \| f \|_{H^k_{[-1,1]}}$ for some constant $c_k(T_0)$ independent of $f$.  Therefore, the result follows.
\end{proof}

\subsection{Exponential convergence}\label{ss:convergence_exact}

\subsubsection{The continuous Fourier extension}\label{sss:expansions}

As mentioned in the introduction, the continuous Fourier extension defined by~\eqref{E:ls} can be characterized in terms of certain nonclassical orthogonal polynomial expansions.  This was demonstrated for the case $T=2$ in \cite{huybrechs2010fourier}, but can be shown for any $T>1$ with relatively minor modifications.

Our approach is to construct an orthogonal basis for the space $G_n$ of $2T$-periodic functions. Since the least squares criterion~\eqref{E:ls} corresponds to an orthogonal projection, it then suffices to expand a given function $f$ in this basis in order to find its Fourier extension.

The cosine and sine functions in $G_n$ are already mutually orthogonal and hence can be treated separately. Consider the cosines first, i.e. the set
\[
 C_n := \{ \CKPTX \}_{k=0}^n.
\]
Trigonometric functions are closely related to polynomials, through an appropriate \emph{cosine-mapping}. This can be seen, for example, from the defining property of Chebyshev polynomials of the first kind $T_k$,
\[
 \cos k x = T_k( \cos x).
\]
In the same spirit, we define the map
\begin{equation}\label{E:map}
 y = \CPTX,
\end{equation}
and note that, since $\cos kx$ is an algebraic polynomial of degree $k$ in $\cos x$, the function $\CKPTX$ is a polynomial in $y$. From this we conclude that the set of cosines $C_n$ is a basis for the space of  algebraic polynomials in $y$ of degree $n$.

Since the functions in $C_n$ are linearly independent, an orthogonal basis exists. Moreover, we may write the basis functions as polynomials in $y$, say $T_k^T(y)=T_k^T(\CPTX)$. Orthogonality in $L^2[-1,1]$ implies
\begin{align*}
\delta_{kl} &= \int_{-1}^1 T_k^T(\CPTX) T_l^T(\CPTX) \DX{x} \\
&= 2 \int_{0}^1 T_k^T(\CPTX) T_l^T(\CPTX) \DX{x} \\
&= \frac{2T}{\pi} \int_{c(T)}^1 T_k^T(y) T_l^T(y) \CHW \DX{y},
\end{align*}
where, for ease of notation, we have defined the $T$-dependent constant
\[
c(T) := \cos \frac{\pi}{T}.
\]
In the latter step above, we applied the substitution~\eqref{E:map}, which maps the interval $[0,1]$ to $[c(T),1]$ and introduces the Jacobian with the inverse square root. It follows that the $T_k^T(y)$ are orthonormal polynomials on $[c(T),1]$ with respect to the weight function
\begin{equation}
\label{E:w1}
 w_1(y) = \frac{2T}{\pi} \CHW.
\end{equation}
This weight differs only by a constant factor from the typical weight function of Chebyshev polynomials of the first kind $T_k(y)$. However, the interval of orthogonality is different from that of the Chebyshev polynomials, since $[c(T),1]$ is contained within $(-1,1]$ for $T>1$, whereas Chebyshev polynomials are orthogonal over the whole interval $[-1,1]$.

We now consider the set of sines in $G_n$,
\[
 S_n := \{ \SKPTX \}_{k=1}^n.
\]
Analogously, this leads to orthogonal polynomials resembling Chebyshev polynomials of the second kind $U_k(y)$.  Indeed, from the property
\[
 \sin (k+1) x = U_k( \cos x) \sin x
\]
we find that $\SKPTX$ is also a polynomial in $y$, but only up to an additional factor. This factor is
\[
 \SPTX= \sqrt{1-y^2}.
\]
We therefore consider an orthogonal basis in the form
\[
U_k^T(y) \sqrt{1-y^2} = U_k^T(\CPTX) \SPTX.
\]
Orthogonality in $L^2 [-1,1]$ implies
\begin{align*}
\delta_{kl} &= \int_{-1}^1 U_k^T(\CPTX) \SPTX U_l^T(\CPTX) \SPTX \DX{x} \\
&= 2 \int_{0}^1 U_k^T(\CPTX) \SPTX U_l^T(\CPTX) \SPTX \DX{x} \\
&= \frac{2T}{\pi} \int_{c(T)}^1 U_k^T(y) U_l^T(y) \CHWS \DX{y}.
\end{align*}
Thus, the $U_k^T(y)$ are again orthonormal polynomials. The weight function
\begin{equation} \label{E:w2}
 w_2(y) = \frac{2T}{\pi} \CHWS
\end{equation}
corresponds to that of the Chebyshev polynomials of the second kind, but here too the interval of orthogonality $[c(T),1]$ differs from the classical case in the same $T$-dependent manner.

Since the functions $T^{T}_{k}(\CPTX)$, $U^T_k(\CPTX) \SPTX$ comprise a basis for the space $G_n$, and since the continuous Fourier extension $g_n$ is the orthogonal projection of $f$ onto $G_n$, we now deduce the following theorem:

\begin{theorem}\label{th:exact_solution}
 The continuous Fourier extension $g_n$ of a function $f$ is precisely
\begin{equation*}
 g_n(x) = \sum_{k=0}^n a_k T_k^T\left( \CPTX \right) + \sum_{k=0}^{n-1} b_k U_k^T\left( \CPTX \right) \SPTX,
\end{equation*}
where
\begin{align}
a_k &= \int_{-1}^1 f(x) T_k^T(\CPTX) \DX{x}, \label{E:ak_x}\\
b_k &= \int_{-1}^1 f(x) U_k^T(\CPTX) \SPTX \DX{x}, \label{E:bk_x}
\end{align}
and the polynomials $T_k^T(y)$ and $U_k^T(y)$ are orthonormal on $[c(T),1]$ with respect to the weights $w_1(y)$ and $w_2(y)$ given by~\eqref{E:w1} and~\eqref{E:w2} respectively.
\end{theorem}

\subsubsection{Exponential convergence of the continuous Fourier extension}\label{sss:convergence_exact}

The functions being expanded in the theory of the continuous Fourier extension above are related to $f$ through the inverse of the map~\eqref{E:map}. It is necessary to distinguish between the even and odd parts of $f$,
\begin{align*}
 f_e(x) = \frac{f(x)+f(-x)}{2}, \quad  f_o(x) = \frac{f(x)-f(-x)}{2}.
\end{align*}
Since
\begin{align}
 a_k &= \int_{-1}^1 f(x) T_k^T(\CPTX) \DX{x} \nonumber \\
 &= 2 \int_{0}^1 f_e(x) T_k^T(\CPTX) \DX{x} \nonumber \\
 &= \frac{2T}{\pi} \int_{c(T)}^1 f_e\left( \frac{T}{\pi} \cos^{-1} y \right) T_k^T(y) \CHW \DX{y},\label{E:ak_y}
\end{align}
we see that the even part of the Fourier extension $g_n$ is precisely the orthogonal polynomial expansion of the function
\[
 f_1(y) = f_e\left( \frac{T}{\pi} \cos^{-1} y \right) = f_e(x).
\]
A similar reasoning for the odd part of $f$, based on the coefficients $b_k$, yields the second function
\[
 f_2(y) = \frac{f_o\left( \frac{T}{\pi} \cos^{-1} y \right)}{\sqrt{1-y^2}} = \frac{f_o(x)}{\SPTX},
\]
with the odd part of $g_n$, when divided by $\sin \frac{\pi}{T} x$, corresponding to the expansion of $f_2$ in the polynomials $U^{T}_{k}(y)$.

Let us now recall some standard theory on orthogonal polynomial (see, for example, \cite{boyd,rivlin1990chebyshev,trefethen2008clenshawcurtis} for an in-depth treatment).  The convergence rate of the expansion of an analytic function $f$ in a set of orthogonal polynomials is exponential, and for a wide class of orthogonal polynomials on the interval $[-1,1]$ the precise rate of convergence is determined by the largest \emph{Bernstein ellipse} 
\[
 e(\rho) := \{ \frac{1}{2} (\rho^{-1} e^{-\ii \theta}+\rho e^{\ii \theta}): \theta \in [-\pi,\pi]\}, \quad \rho \geq 1,
\]
within which $f$ is analytic.  We shall use the convention $\rho \geq 1$ throughout.  Note that the ellipse $e(\rho)$ has foci $\pm 1$, and the parameter $\rho$ coincides with the sum of the lengths of its major and minor semiaxes.  Suppose now that $f$ is analytic within the Bernstein ellipse of radius $\rho = \rho_{\max}$ and not analytic in any ellipse with $\rho > \rho_{\max}$.  Then the rate convergence of the expansion of $f$ in a set of orthogonal polynomials on $[-1,1]$ is precisely $(\rho_{\max})^{-n}$.

Returning to the continuous Fourier extension, we now see that its convergence rate is determined by the nearest singularity of $f_1(y)$ or $f_2(y)$ to the interval $[c(T),1]$. Note that, even when $f$ is entire, singularities are introduced by the inverse cosine mapping at the points $y= \pm 1$. One can verify that the singularity at $y=1$ is removable for both $f_1$ and $f_2$. Thus, the nearest singularity lies at $y=-1$.

To apply the above polynomial theory, it remains to transform the interval $[c(T),1]$ to the standard interval $[-1,1]$. This is achieved by the affine map
\begin{equation}\label{E:m}
m^{-1}(s) = 2\frac{s-c(T)}{1-c(T)} - 1,\quad s \in [c(T),1],
\end{equation}
with inverse $m(t) = \frac{1}{2}(1-c(T))t+\frac{1}{2}(1+c(T))$ mapping $[-1,1]$ to $[c(T),1]$.  Note that $m^{-1}$ maps the point $y=-1$ to the point
\[
 u = \frac{ 3 + c(T)} {c(T)-1} = 1 - 2\mathrm{cosec}^2 \left ( \tfrac{\pi}{2 T} \right ) <-1,
\]
which lies on the negative real axis. Since the Bernstein ellipse $e(\rho)$ crosses the negative real axis in the point $-\frac12 (\rho^{-1}+\rho)$, equating this with $u$ yields the maximal value of $\rho$.  We therefore deduce the following result:
\begin{theorem}\label{th:expconv}
 For all sufficiently analytic functions $f$, the error in approximating $f$ by its continuous Fourier extension $g_n$ satisfies
\begin{equation} \label{E:error}
 | f(x) - g_n(x) | \leq c_f E(T)^{-n},\quad \forall n \in \mathbb{N},\ T > 1,
\end{equation}
uniformly for $x \in [-1,1]$, where $c_f$ is a constant depending on $f$ only, and 
\[
 E(T) = \cot^2 \left ( \tfrac{\pi}{4 T} \right ).
\]
\end{theorem}
This theorem extends Theorem 3.14 of \cite{huybrechs2010fourier} to the case of general $T > 1$.  Note that the estimate~\eqref{E:error} holds for all $n\in \mathbb{N}$ and $T>1$.  For the sake of brevity, we have limited the exposition here to the case of sufficiently analytic functions $f$. Yet, we make the following comments:
\begin{itemize}
 \item The rate of exponential convergence may be slower than $E(T)$ when $f(x)$ is not sufficiently analytic as a function of $x$, i.e. if $f$ has a singularity closer than that introduced by the inverse cosine mapping.  
  \item The convergence may also be faster than $E(T)$ if $f$ is analytic and periodic on $[-T,T]$. In that case, one can verify that the singularity introduced by the inverse cosine at $y=-1$ is also removable. Thus, the convergence rate is no longer limited by the singularity introduced by the map between the $x$ and $y$ domains.
 \item For the case $T=2$ we recover the convergence rate $E(2) = 3+2\sqrt{2}$ found in~\cite[\S3.4]{huybrechs2010fourier}.
\end{itemize}
	
The function $E(T)$ is depicted in Fig.~\ref{fig:theory}. It is monotonically increasing on $(1,\infty)$ as a function of $T$, and behaves like $1+\pi (T-1)$ for $T \approx 1$ and $\frac{16}{\pi^2}T^2$ for $T \gg 1$. Thus, larger $T$ leads to more rapid exponential convergence.  This is confirmed in Fig.~\ref{fig:convrate}, where we plot the error in approximating $f(x) = e^x$ by $f_n$ for various $n$ and $T$ (for this example we use high precision to avoid any numerical effects -- see \S \ref{ss:convergence_numerical}).  Note the close correspondence between the observed and predicted convergence rates, as well as the significant increase in convergence rate for larger $T$.  Having said this, it transpires that increasing $T$ adversely affects the resolution power, a topic we will consider further in \S\ref{s:resolution}.

\begin{figure}[t]
\begin{center}
    \includegraphics[width=5cm]{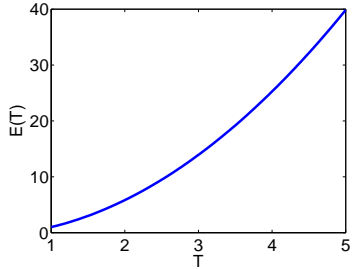}
 \caption{Theoretical convergence rate $E(T)$ of the continuous Fourier extension as a function of the extension parameter $T$.}\label{fig:theory}
\end{center}
\end{figure}

\begin{figure}[t]
\begin{center}
  \subfigure[Log error]{
    \includegraphics[width=5cm]{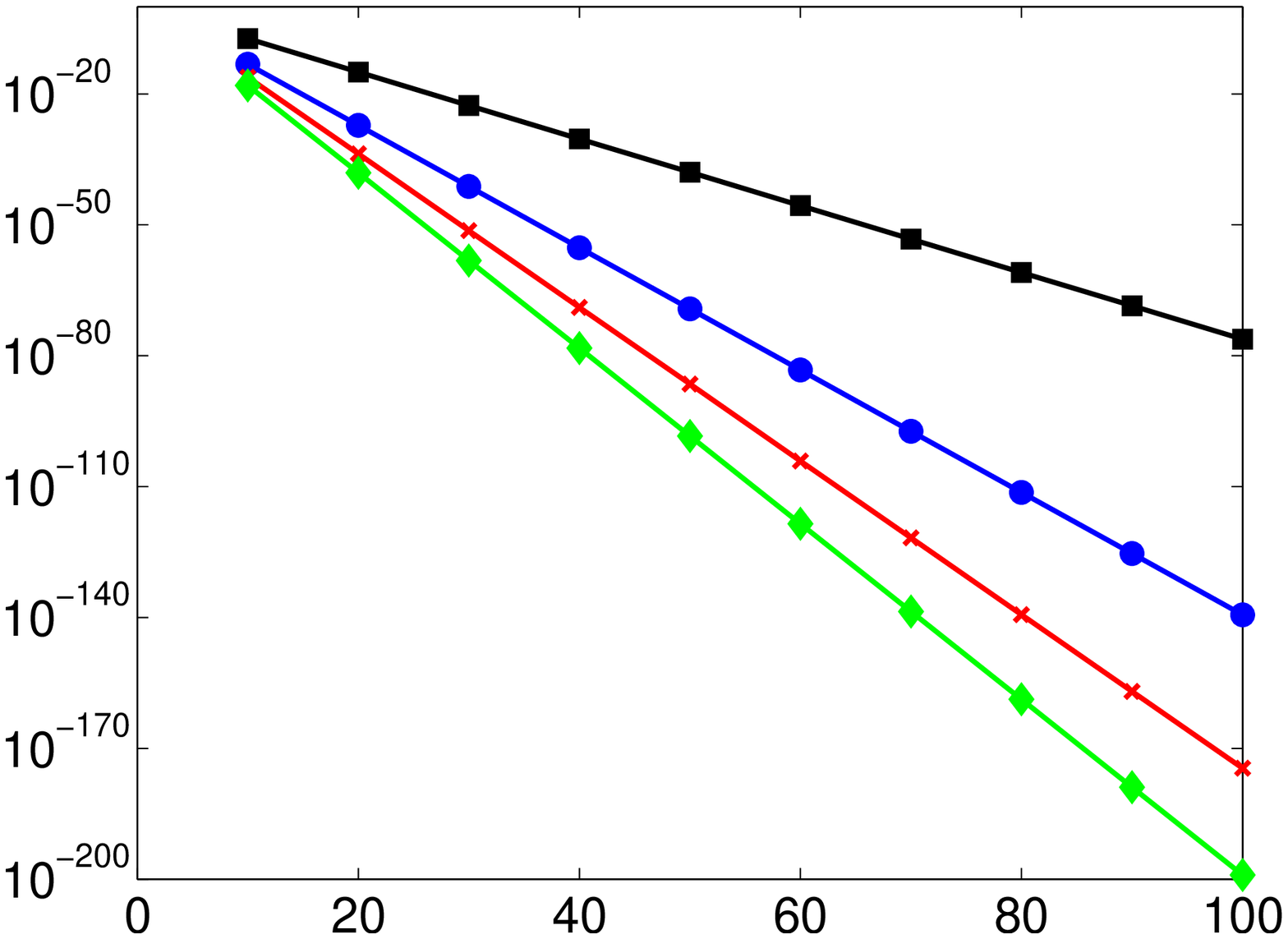}
  }
  \subfigure[Convergence rate]{
    \includegraphics[width=5cm]{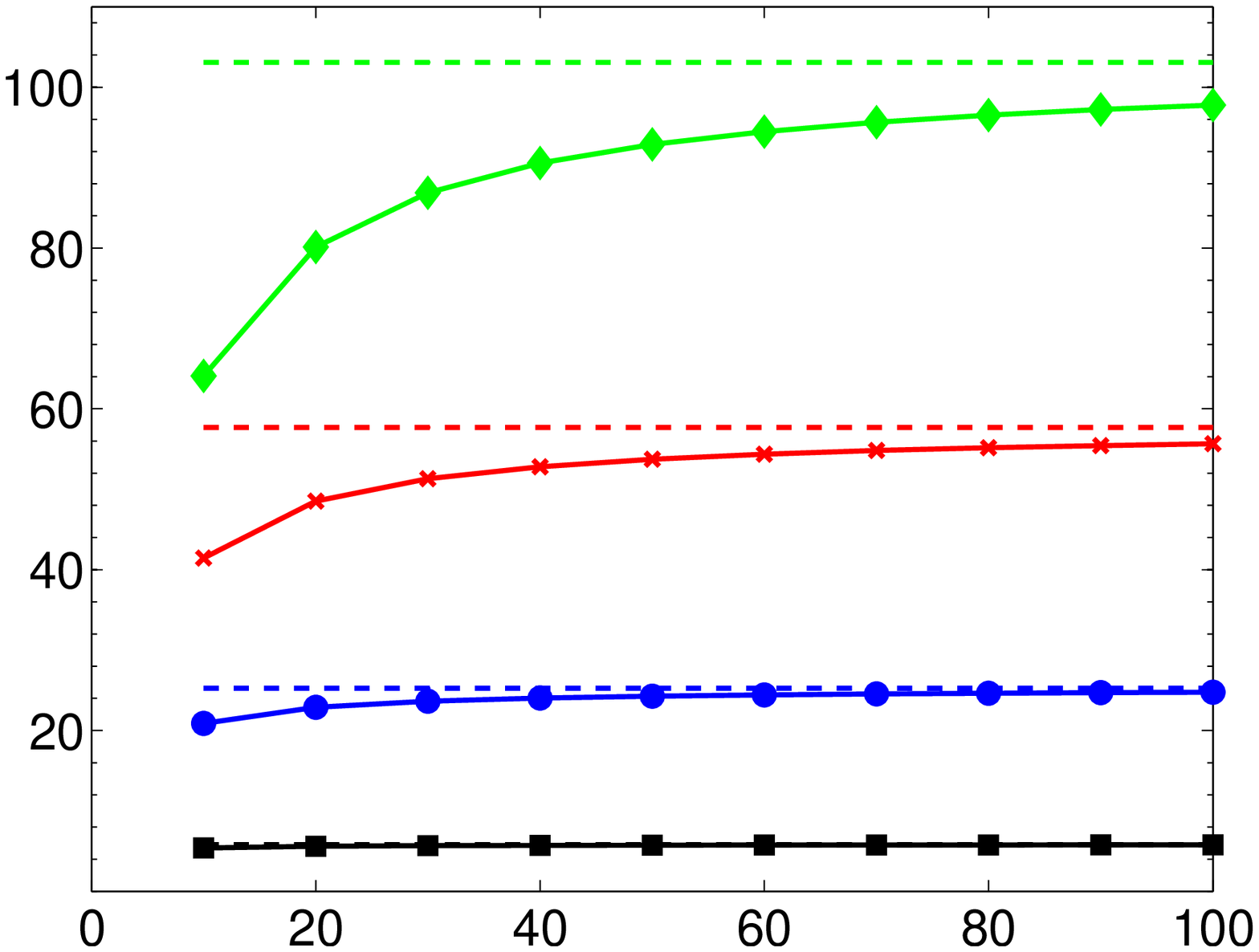}
  }
 \caption{The quantities $e_{n} = \| f - g_n \|_{L^\infty_{[-1,1]}}$ (solid line) and $E(T)^{-n}$ (dashed line) for $T=2,4,6,8$.  The left panel shows the values $e_n$ and $E(T)^{-n}$.  The right panel gives the scaled values $\exp(- \frac{1}{n} \log e_n )$ and $E(T)$.}\label{fig:convrate}
\end{center}
\end{figure}

\section{Computation of Fourier extensions}\label{ss:convergence_numerical}
Although in \S \ref{sss:expansions} we introduced a new, orthogonal basis for the set $G_n$, this basis is never actually used in computations.  Instead, we always use either complex exponentials or real sines and cosines.  The reason for this is primarily simplicity (recall that the relevant orthogonal basis relates to nonstandard orthogonal polynomials, which therefore would need to be precomputed), and the fact that having a Fourier extension in the form of a standard Fourier series is certainly most convenient in practice. 

With this in mind, let $\{ \phi_k \}^{2n+1}_{k=1}$ be an enumeration of a `standard' basis for the set $G_n$ (e.g. complex exponentials), and suppose that the continuous Fourier extension $g_n$ of a function $f$ has coefficients $\{ a_k \}^{2n+1}_{k=1}$ in this basis.  Since $g_n$ is the solution to~\eqref{E:ls}, it is straightforward to show that
\begin{equation} \label{E:exactlinsys}
A a = B,
\end{equation}
where $A \in \C^{(2n+1) \times (2n+1)}$ and $B \in \C^{2n+1}$ have entries $\langle \phi_k , \phi_j \rangle$ and $\langle f , \phi_j \rangle$ respectively, and $a$ is the vector of coefficients $a_k$.  Thus, given $B$, we may compute the continuous Fourier extension $g_n$ by solving~\eqref{E:exactlinsys}.

There are two issues with computing the continuous Fourier extension.  First, one requires knowledge (or prior computation) of the integrals $\langle f , \phi_j \rangle$.  Second, the condition number $\kappa(A) \sim E(T)^{2n}$ \cite{huybrechs2010fourier}.  Such severe ill-conditioning typically limits the maximal achievable accuracy (by this we mean the smallest possible error $\| f - g_n \|$ that can be attained in finite precision for any $n$) of the continuous Fourier extension to $\BIGO(\sqrt{\epsilon})$, where $\epsilon$ is the machine precision used \cite{huybrechs2010fourier}.

Ill-conditioning of the exact extension is due to the redundancy of the set $G_\infty$.  This is explained in the next section.  Fortunately, its effect can be greatly mitigated by computing a so-called \emph{discrete} Fourier extension instead.  This extension, which converges at precisely the same rate as the continuous Fourier extension, involves only pointwise evaluations of $f$, and consequently also avoids the first issue stated above.  This is discussed in \S \ref{ss:discrete}.

\subsection{Ill-conditioning}
The reason for ill-conditioning in the exact Fourier extension is simple: although the functions $\exp(i \frac{k \pi}{T} x)$ form an orthogonal basis on $[-T,T]$, they only constitute a frame when restricted to the interval $[-1,1]$.

\begin{lemma}\label{lem:frame}
The set 
\begin{equation} \label{E:fourframe}
\Phi : = \{ \frac{1}{\sqrt{2}} \} \cup \{ \frac{1}{\sqrt{2}}\exp( \ii \tfrac{k \pi}{T} x ) \}_{k \in \Z \backslash \{ 0 \} },
\end{equation}
is a normalized tight frame for $L^2[-1,1]$ with frame bound $T$.
\end{lemma}
\begin{proof}
Let $f \in L^2[-1,1]$ be given and define $\tilde f \in L^2[-T,T]$ as the extension of $f$ by zero to $[-T,T]$.  Since
\begin{equation}
\label{E:akdef}
a_k=\frac{1}{\sqrt{2}} \int^{1}_{-1} f(x) \exp( \ii \tfrac{k \pi}{T} x ) \DX{x} = \frac{1}{\sqrt{2}} \int^{T}_{-T} \tilde f(x) \exp( \ii \tfrac{k \pi}{T} x ) \DX{x},
\end{equation}
we find that $\frac{1}{\sqrt{T}}a_k$ is precisely the $k$th Fourier coefficient of $\tilde f$ on $[-T,T]$.  By Parseval's relation
\[
\sum^{\infty}_{k=-\infty} | a_k |^2 = T \| \tilde f \|^2 = T \| f \|^2,
\]
as required.
\end{proof}
\noindent For a general introduction to the theory of frames, see~\cite{christensen2003introduction}.

Lemma \ref{lem:frame} has several implications. Recall that frames are usually redundant: any given function $f$ typically has infinitely many representations in a particular frame.  This implies that there will be approximate linear dependencies amongst the columns of the Gram matrix $A$ for all large $n$.  This makes $A$ ill-conditioned, and in the case of the frame~\eqref{E:fourframe}, leads to the aforementioned exponentially large condition numbers.

Let us consider several particular representations of a smooth function $f$ in the frame~\eqref{E:fourframe}.  Recall that associated to any frame $\{ f_k \}$ is the so-called canonical dual frame $\{ g_k \}$ \cite{christensen2003introduction}.  A representation of $f$, the \textit{frame decomposition}, is given in terms of this frame by
\[
f = \sum^{\infty}_{k = 1} \langle f , g_k \rangle f_k.
\]
Since the frame~\eqref{E:fourframe} is tight, it coincides with its canonical dual up to a constant factor that is equal to the frame bound, in this case $T$. Therefore its frame decomposition is precisely
\[
 \sum^{\infty}_{k=-\infty} \frac{1}{\sqrt{T}}a_k \frac{\exp( \ii \tfrac{k \pi}{T} x )}{\sqrt{2 T}} ,
\]
where the $a_k$'s are given by~\eqref{E:akdef}.  However, as illustrated in the proof of the previous lemma, this is precisely the Fourier series of the discontinuous function $\tilde f$, and thus this infinite series converges only slowly and suffers from a Gibbs phenomenon.  

On the other hand, as shown in the proof of Theorem~\ref{t:specconv}, by extending $f$ more smoothly to $[-T,T]$, one obtain representations of $f$ in the frame~\eqref{E:fourframe} that converge algebraically at arbitrarily fast rates.  Thus, we conclude the following: different frame expansions of the same function may give rise to wildly different approximations.

Clearly, given its exponential rate of convergence, the continuous Fourier extension $g_n$ will not coincide with any of the aforementioned representations (recall also that $g_n$ does not typically converge outside $[-1,1]$ \cite{huybrechs2010fourier}, and therefore its limit cannot be a frame expansion in general). More importantly, however, it is not certain that the solution of the linear system~\eqref{E:exactlinsys}, when computed in finite precision, will actually coincide with the continuous Fourier extension $g_n$ itself.  The reason is straightforward: for large $n$, $A$ is approximately underdetermined, and therefore there will be many approximate solutions to~\eqref{E:exactlinsys} corresponding to different frame expansions of $f$.  A typical linear solver will usually select an approximate solution with bounded coefficient vector $a$, and there is no guarantee that this need coincide with $g_n$.

This phenomenon has a significant potential consequence: since the numerically computed extension need not coincide with the continuous Fourier extension $g_n$, theoretical results for the behaviour of $g_n$ are not guaranteed to be inherited by the numerical solution.  Having said this, in numerical experiments, one typically witnesses exponential convergence, exactly as predicted by Theorem \ref{th:expconv}.  However, a difference can arise when investigating other properties, such as resolution power, as we discuss and explain in \S \ref{s:resolution}.

The above discussion is not intended to be rigorous.  A full analysis of the behaviour of `numerical' Fourier extensions is outside the scope of this paper.  It transpires that this can be done, and a complete analysis will appear in a future paper \cite{BADHJMVFEStability}.  Instead, in the remainder of this paper, we focus mainly on approximation-theoretic properties of theoretical extensions; resolution power, in particular.  Nonetheless, in \S \ref{ss:resolution_numerical} we revisit the issue of numerical extensions insomuch as it relates to resolution, and give an explanation of the phenomenon of  differing resolution power mentioned above.

\subsection{A discrete Fourier extension}\label{ss:discrete}
The continuous Fourier extension is defined by a least squares criterion, and as such it suffers from the drawback of requiring computation of the integrals $\langle f , \phi_k \rangle$.  To avoid this issue one may instead define a Fourier extension by a collocation condition.  In other words, if $\{ x_i \}^{2n+1}_{i=1}$ is a set of nodes in $[-1,1]$ we replace the matrix $A$ and the vector $B$ in~\eqref{E:exactlinsys} by $\tilde A$ and $\tilde B$ with entries $\phi_k(x_j)$ and $f(x_j)$ respectively.  We then define $\tilde A a = \tilde B$ once more.

The key question is how to determine good nodes.  Recall that $g_n$, as defined by~\eqref{E:ls}, is essentially a sum of two polynomial approximations in the variable $y \in [c(T),1]$.  The polynomial interpolant of an analytic function in $y$ at Chebyshev nodes (appropriately scaled to the interval $[c(T),1]$) converges exponentially fast at the same rate $\rho$ as its orthogonal polynomial expansion.  Therefore, such nodes, when mapped to the original domain via $x = \frac{T}{\pi} \cos^{-1} y$, will ensure exponential convergence at rate $E(T)$ of the resulting Fourier extension.

It is now slightly easier to redefine $G_n$ to be the space of dimension $2n+2$ spanned by the functions $\{ \CKPTX \}^{n}_{k=0}$ and $\{ \SKPTX \}^{n+1}_{k=1}$.  The $2n+2$ collocation nodes in $[-1,1]$ therefore take the form $\{ x_i \}^{n}_{i=0} \cup \{ -x_i \}^{n}_{i=0}$, where
\begin{equation}\label{E:chebnodes}
x_i = \frac{T}{\pi} \cos^{-1} \left [ \frac{1}{2} (1-c(T)) \cos \left ( \frac{(2i+1) \pi}{2n+2} \right ) + \frac{1}{2}(1 + c(T)) \right ],\quad i=0,\ldots,n.
\end{equation}
Recall that $c(T) = \cos \frac{\pi}{T}$.  We refer to such nodes as \textit{symmetric mapped Chebyshev} nodes, and the corresponding extension $g_n$ as the \textit{discrete} Fourier extension of the function $f$.

Besides removing the requirement to compute the integrals $\langle f , \phi_n \rangle$, this choice of nodes also has the significant effect of improving conditioning, as we now explain.  We first note the following:
\begin{lemma}\label{lem:collocation_galerkin}
Let $\tilde A$ be the collocation matrix based on the symmetric mapped Chebyshev nodes~\eqref{E:chebnodes}, and suppose that $D$ is the diagonal matrix of weights $\omega_i = \frac{\pi}{n+1}$.  Then the matrix $A_W = \tilde A^{\top} D \tilde A$ has entries 
\[
\langle \phi_k , \phi_j \rangle_{W} = \int^{1}_{-1} \phi_{j}(x) \phi_{k}(x) W(x)  \DX{x},\quad j,k=1,\ldots,2n+2,
\]
where $W$ is the positive, integrable function given by
\[
W(x) = \frac{2 \pi}{T}  \frac{ \cos \frac{\pi}{2T} x} { \sqrt{ \CPTX - \CPT }}.
\]
\end{lemma}
\begin{proof}
Note that $A_{W}$ is a block matrix with blocks corresponding to inner products of sines and cosines with sines and cosines.  Consider the upper left block of the matrix $\tilde{A}^{\top} D \tilde{A}$.  In the $(j,k)$th entry, using the symmetry of the collocation points, we have that 
\[ 
2 \sum^{n}_{i=0} \omega_i \phi_{j}(x_i) \phi_k(x_i) = 2 \sum^{n}_{i=0} \omega_i T_j(y_i) T_k(y_i),
\]
where $\{ y_i \}^{n}_{i=0}=\cos(\frac{\pi}{T} x_i)$ are Chebyshev nodes on $[c(T),1]$.  Recall that $m(t)$ is the affine map from $[-1,1]$ to $[c(T),1]$, as defined in \S\ref{sss:convergence_exact}.
Since the product $T_j(y) T_k(y)$ is a polynomial in $y$ of degree at most $2n$, the Gaussian quadrature rule associated with the points $x_i$ is exact and it follows that
\begin{align*}
 2\sum^{n}_{i=0} \omega_i \phi_{j}(x_i) \phi_k(x_i) &= 2\int^{1}_{-1} T_j(m(t)) T_k(m(t)) \frac{1}{\sqrt{1-t^2}} \DX{t} \\
&= 2 \int^{1}_{c(T)} T_j(y) T_k(y) w(y) \DX{y},
\end{align*}
where $w(y)$ is given by
\[
w(y) = \frac{2}{1-c(T)} \frac{1}{\sqrt{ 1 - \left ( \frac{2y - 1 - c(T)}{1-c(T)} \right )^2 }}.
\]
The first factor is the Jacobian of the mapping to $[c(T),1]$, the second factor is the scaling of the standard Chebyshev weight to that interval. The change of variables $y = \CPTX$ now gives 
\[
2 \sum^{n}_{i=0} \omega_i \phi_{j}(x_i) \phi_k(x_i) = 2 \frac{\pi}{T} \int^{1}_{0} \phi_j(x) \phi_k(x) w(\cos \tfrac{\pi}{T} x) \sin \tfrac{\pi}{T} x   \DX{x},
\]
which is easily found to coincide with $\langle \phi_k , \phi_j \rangle_{W}$.  The case corresponding to the sine functions $\SKPTX$ is identical.
\end{proof}

This lemma demonstrates that the normal equations for the discrete Fourier extension are the equations of a continuous Fourier extension based on the weighted $L^2_{W}$ optimization problem.  It is  well-known that forming normal equations leads to worse numerical results \cite[Ch.19]{lawson1996leastsquares}, and exactly the same is true in this case.  Indeed, since $W$ is an integrable weight function, we may expect that $\kappa(A_W) \approx E(T)^{2n}$, i.e. exactly as in the unweighted case.  Thus collocation leads to the significant reduction in condition number, with $\kappa(\tilde A) \approx E(T)^n$ as opposed to $E(T)^{2n}$.  As a result, one typically observes a much higher accuracy, $\BIGO(\epsilon)$ as opposed to $\BIGO(\sqrt{\epsilon})$, with this approach.

This improvement is illustrated in Fig.~\ref{fig:methods} for the example $f(x) = \cos 16 x$.  As shown in the left panel, the best attainable error with the continuous Fourier extension is roughly $10^{-8}$, whereas the discrete extension obtains much closer to machine epsilon.  The right panel indicates the significantly milder growth in condition number.

\begin{figure}[t]
\begin{center}
  \subfigure[Approximation error]{
    \includegraphics[width=5cm]{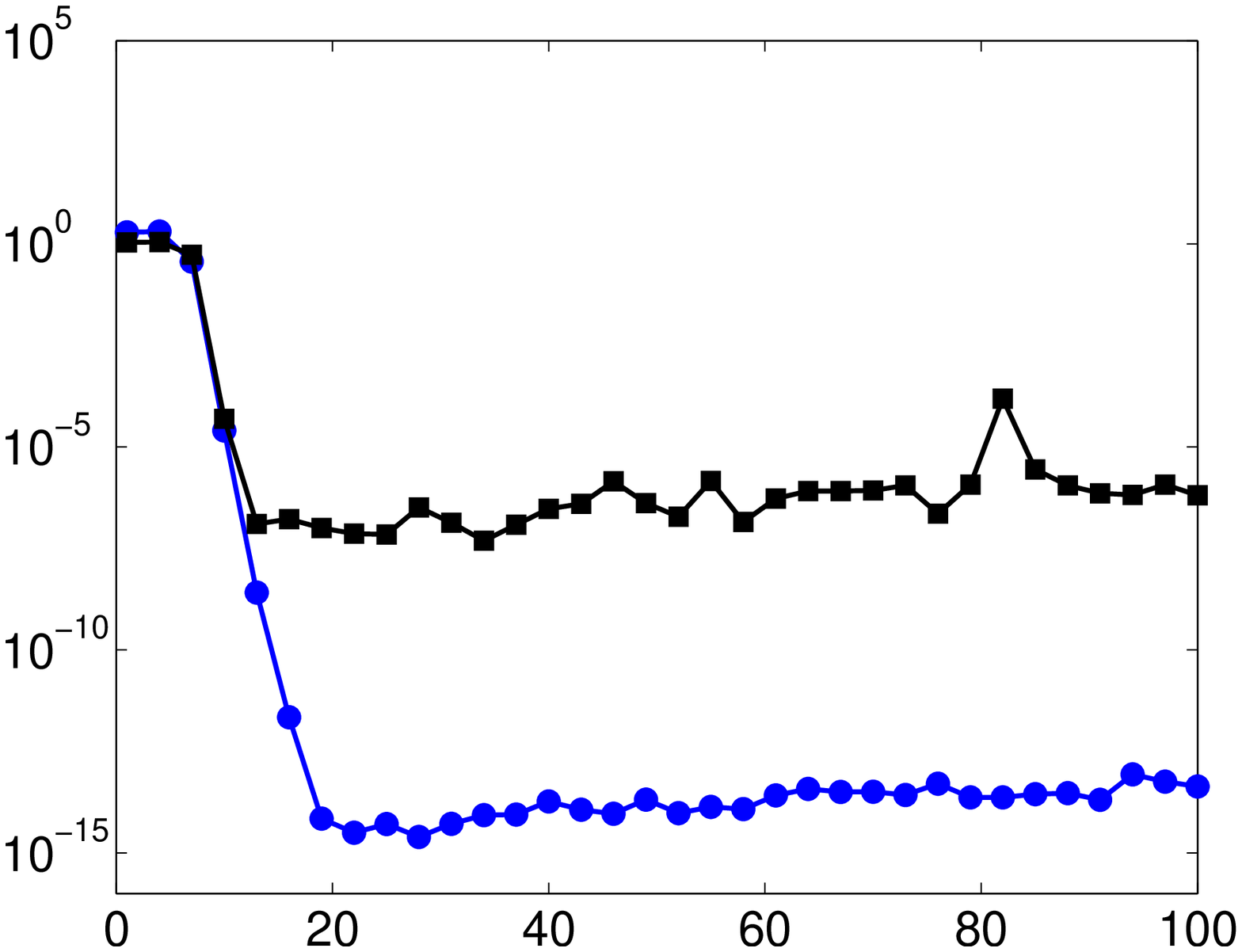}
  }
  \subfigure[Condition number]{
    \includegraphics[width=5cm]{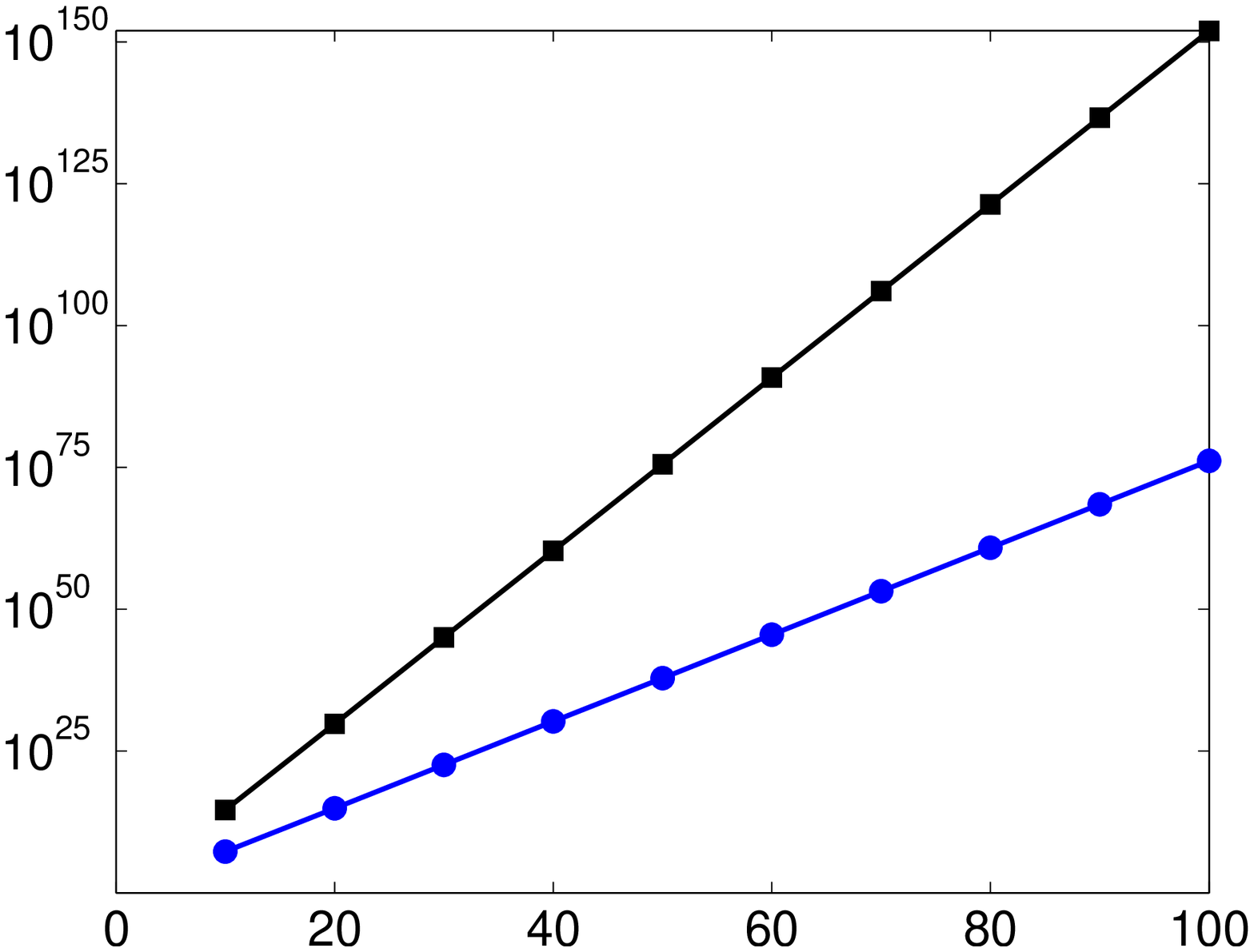}
  }
 \caption{The approximation errors $\| f - g_n \|_{L^\infty_{[-1,1]}}$ (left panel) and condition numbers $\kappa(A)$ (right panel) for the continuous (squares) and discrete (circles) Fourier extensions with $T=2$ and $f(x)=\cos 16 x$.}\label{fig:methods}
\end{center}
\end{figure}

Such an improvement is by no means unique to this choice of nodes.  A suitable alternative choice of collocation points follows from using the roots of the polynomials $T_{k}^T(y)$. With an appropriate number of points, the matrix $\tilde{A}^T D \tilde{A}$, with $D$ containing the weights of the associated Gaussian quadrature rule on its diagonal, is precisely the matrix of the unweighted continuous Fourier extension, rather than the weighted extension that was identified in Lemma~\ref{lem:collocation_galerkin}.  The use of these points as Gaussian quadrature was already explored in~\cite{huybrechs2010fourier}. However, these points depend on $T$ in a non-trivial way and, although they can easily be computed, are not available in closed form. As numerical experiments indicate that the performance of such point sets is not significantly better than the points proposed in this section, we forego a more detailed analysis. 

This aside, we remark that, in some senses, the continuous and discrete Fourier extensions are analogous to the orthogonal expansion of a function in Chebyshev polynomials and its Chebyshev polynomial interpolant (sometimes known as the discrete Chebyshev expansion).  Specifically, the discrete versions both arise by essentially replacing an inner product by a quadrature.  However, the correspondence is not exact, since in the discrete Fourier extension the quadrature approximates the weighted inner product on $L^2_W[-1,1]$.  As mentioned, this is done for computational convenience. Nonetheless, the approximation-theoretic properties of the discrete and continuous Fourier extensions are nearly identical, with the only key difference being that of the condition number.

We have purposefully kept this section on numerical computation of the continuous and discrete Fourier extensions short.  The main conclusion is that, despite quite severe ill-conditioning (even in the discrete case we have exponentially large condition numbers), one can typically obtain very high accuracies (i.e. close to machine precision) with the discrete Fourier extension.  Although this apparent contradiction can be comprehensively explained, it is beyond the scope of this paper, and will be reported in a subsequent work \cite{BADHJMVFEStability}.

\section{Resolution power of polynomial expansions}
\label{s:OPSresolution}
Having introduced and analyzed the continuous and discrete Fourier extensions, and discussed their numerical implementation, in \S\ref{s:resolution} we establish their resolution power.  Before doing so, since similar techniques will be used subsequently, we first derive the well-known result for orthogonal polynomial expansions.

Suppose that $f(x)$ is defined on $[-1,1]$ and is analytic in a neighbourhood thereof.  Let $f_{n} \in \mathbb{P}_{n}$ be its $n$-term expansion in orthogonal polynomials with respect to the positive and integrable weight function $w(x)$.  If $L^{2}_{w}[-1,1]$ is the space of weighted square integrable functions with respect to $w$, then $f_n$ is given by
\begin{equation*}
f_n : =  \underset{p \in \mathbb{P}_n}{\operatorname{argmin}}  \Vert f-p \Vert_{L^2_{[-1,1],w}}.
\end{equation*}
In particular, for any $p_n \in \mathcal{P}_{n}$, 
\begin{equation*}
\Vert f-f_n \Vert^2_{L^2_{[-1,1],w}} \leq \Vert f-p_n \Vert^2_{L^2_{[-1,1],w}} \leq \| w \|_{L^1_{[-1,1]}} \Vert f-p_n \Vert^2_{L^\infty_{[-1,1]}} .
\end{equation*}
Hence, we note that to study of the resolution power of any orthogonal polynomial expansion it suffices to consider only one particular example.  Without loss of generality, we now focus on expansions in Chebyshev polynomials (i.e. $w(x) = (1-x^2)^{-\frac{1}{2}}$), in which case
\begin{equation}\label{E:chebyshev_expansion}
p_n(x) = \sum^{n}_{k=0} ~' a_k T_k(x),\qquad f(x) = \sum_{k=0}^\infty ~' a_k T_k(x),
\end{equation}
where
\begin{equation}\label{E:expansion_coefficients}
 a_k= \int_0^\pi f(\cos \theta) \cos k \theta \DX{\theta},
\end{equation}
and $'$ indicates that the first term of the sum should be halved.  Since the infinite sum (\ref{E:chebyshev_expansion}) converges uniformly, we have the estimate
\[
\Vert f-p_n \Vert_{L^\infty_{[-1,1]}} \leq \sum_{k > n} | a_k |.
\]
Therefore it suffices to examine the nature of the coefficients $a_n$ for the function
\begin{equation}\label{E:test_f}
 f(x) = \exp( \ii \pi \omega x).
\end{equation} 
We use the following standard estimate, given in~\cite[p.175]{rivlin1990chebyshev}:
\begin{equation}\label{E:chebyshev_bound}
 |a_n| \leq \frac{2M_{\rho}}{\rho^n}.
\end{equation}
Here $\rho$ corresponds to any Bernstein ellipse $e(\rho)$ in which $f$ is analytic and $M_{\rho}$ is the maximum of $|f(z)|$ along that ellipse.

For a given $\omega$ and $n$, we consider the minimum of all bounds of the form~\eqref{E:chebyshev_bound}.  Since $f$ is a complex exponential, $f$ reaches a maximum at the point on the ellipse $e(\rho)$ with the smallest (negative) imaginary part. This corresponds to $\theta=-\pi/2$ and thus we have
\begin{equation}
 M_\rho = \exp\left( \pi \omega \frac{\rho^2-1}{2\rho} \right).
\end{equation}
The bound~\eqref{E:chebyshev_bound} now becomes
\begin{equation}\label{E:resolution_bound}
 B(\omega,n,\rho) := \frac{2M_\rho}{\rho^n} = \frac{2\exp\left( \pi \omega \frac{\rho^2-1}{2\rho} \right)}{\rho^n}.
\end{equation}
For fixed $\omega$ and $n$, we find the minimal value of this bound by differentiating~\eqref{E:resolution_bound} with respect to $\rho$. Equating the derivative to zero
\begin{equation}\label{E:rootfinding}
 \frac{d}{d\rho} B(\omega,n,\rho) = 0,
\end{equation}
yields two solutions
\begin{equation}\label{E:solution}
 \frac{n \pm \sqrt{n^2-\pi^2 \omega^2}}{\pi \omega}.
\end{equation}
Consider the case $n < \pi \omega$. Both solutions of~\eqref{E:rootfinding} are complex-valued. Note that
\[
B(\omega,n,1)=2,
\]
and
\[
 \frac{d}{d\rho} B(\omega,n,1) = 2\pi \omega - 2n.
\]
It follows that $B$ is strictly increasing as a function of $\rho$. For $n < \pi \omega$, the best possible bound for $a_n$ of the form~\eqref{E:chebyshev_bound} is $2$.

Consider next the case $n > \pi \omega$. Since it is easily seen that the roots~\eqref{E:solution} are inverses of each other, we restrict our attention to the one greater than $1$, i.e.,
\[
 \rho_{min} := \frac{n + \sqrt{n^2-\pi^2 \omega^2}}{\pi \omega}.
\]
Since $B(\omega,n,\rho)$ initially decays at $\rho=1$, $\rho_{min}$ is the unique minimum of the bound for $\rho > 1$. Thus,
\[
 B(\omega,n,\rho_{min}) < B(\omega,n,\rho), \qquad \rho \in [1,\infty).
\]
In particular, this implies exponential decay of the next coefficients:
\[
 |a_{n+k}| \leq B(\omega,n,\rho_{min}) \frac{1}{\rho_{min}^k} < \frac{2}{\rho_{min}^k}.
\]
Finally, consider the value $n=\pi \omega$. Then $\rho=1$ is a global minimum of the bound~\eqref{E:resolution_bound}, since 
\[
 \frac{d}{d\rho} B(\omega,\pi \omega,1) = \frac{d^2}{d\rho^2} B(\omega,\pi \omega,1) = 0,
\]
and
\[
 \frac{d^3}{d\rho^3} B(\omega,\pi \omega,1) = 2\pi \omega > 0.
\]
In conclusion, we have seen that for an oscillatory function of the form~\eqref{E:test_f}, we need $n=\pi \omega$ degrees of freedom before exponential decay of the coefficients $a_n$ sets in. This corresponds to $\pi$ degrees of freedom per wavelength -- the well-known result on the resolution power of polynomials.

Note that this value was previously derived in \cite{GottGibbs2} for orthogonal polynomial expansions corresponding to the Gegenbauer weight $w(x) = (1-x^2)^{\lambda-\frac{1}{2}}$, $\lambda > - \frac{1}{2}$ (see also \cite[p.35]{naspec} for $\lambda = 0$).  This result was based on explicit expressions for the Gegenbauer polynomial coefficients of $e^{\ii x t}$. The previous arguments generalize this result to arbitrary weight functions $w(x)$.    In fact, we could have also used the explicit formula for the Chebyshev polynomial expansion of $e^{\ii x t}$ to derive estimates similar to those given above (and possibly more accurate).  However, we shall use similar techniques to those presented here in the next section, where such explicit expressions are not available.

\section{Resolution power of Fourier extensions}\label{s:resolution}
We now consider the resolution power of the continuous and discrete Fourier extensions. As commented in \S \ref{ss:convergence_numerical}, the continuous/discrete Fourier extension $g_n$ need not be realized in a finite precision numerical computation.  Hence, we divide this section between theoretical estimates for the resolution power of the continuous/discrete extension, and its numerical realization.

\subsection{Resolution power of the continuous Fourier extension}\label{ss:resolution_exact}
A na\"ive estimate for the resolution constant $r(T)$ follows immediately from the bound (\ref{specineq}) in Theorem \ref{t:specconv}.  Indeed, for $f(x) = \exp(i \omega \pi x)$ we have
\[
\| f \|_{H^{k}_{[-1,1]}} = \BIGO{(\omega \pi)^{k}},\quad k=1,2,\ldots ,
\]
and therefore $r(T)$ satisfies $r(T) \leq 2T$, with spectral convergence occurring once $n$ exceeds $\omega T$.  

On first viewing, this estimate seems plausible.  For example, consider the special situation where the frequency of oscillation $\omega$ is an integer multiple of $T^{-1}$.  Then the function $f(x)= \exp(\ii \frac{\pi }{T} m x)$ is precisely the $m$th complex exponential in the Fourier basis on $[-T,T]$.  Given that the continuous Fourier extension $g_n$ of $f$ is error minimizing amongst all functions in $G_n$, we can expect $f$ to be recovered perfectly (i.e. $g_n \equiv f$) by its continuous Fourier extension whenever $n \geq m = \omega T$.  Thus, for oscillations at frequencies $\omega = \frac{m}{T}$, $m \in \Z$, the estimate $r(T) \leq 2T$ appears correct.  

However, empirical results indicate that such an estimate is only accurate for small $T$: for large $T$ it transpires that $r(T) \sim \pi$.  We shall prove this result subsequently.  Before doing so, however, let us note the following unexpected conclusion: when $T$ is large, we can actually resolve the oscillatory exponentials $\exp(\ii \frac{\pi}{
T} m x)$ accurately on $[-1,1]$ using Fourier extensions comprised of relatively non-oscillatory exponentials $\exp(\ii \frac{\pi}{T} n x)$, $n \ll m$.

 To obtain an accurate estimate for $r(T)$, we need to argue along the same lines as \S\ref{ss:convergence_exact} and exploit the close relationship between Fourier extensions and certain orthogonal polynomial expansions.  Recall that the theory of \S\ref{ss:convergence_exact} treats even and odd cases separately.  Let us first assume that $f$ is even, so that
\[
f(x) = \cos \omega \pi x
\]
(we consider the odd case later).  Upon applying the transformation $x = \frac{T}{\pi} \cos^{-1} y$ we obtain
\[
f_{1}(y) = \cos \omega T \cos^{-1} y,\quad y \in [c(T),1].
\]
The Fourier extension of $f(x)$ is precisely the expansion of $f_{1}(y)$ in the orthogonal polynomials $T^{T}_{k}(y)$.  If we now map the domain $[c(T),1]$ to $[-1,1]$ via
\[
u = \frac{c(T)+1-2y}{c(T)-1},\quad y = \frac{u}{2}(1-c(T)) + \frac{1}{2}(1+c(T)),
\]
then this equates to an orthogonal polynomial expansion of the function
\begin{equation}
\label{E:f2udef}
f_{2}(u) = \cos \left [ \omega T \cos^{-1} \left ( \frac{u}{2}(1-c(T)) + \frac{1}{2}(1+c(T)) \right ) \right ],\quad u \in [-1,1].
\end{equation}
Thus, as in \S\ref{s:OPSresolution}, to determine the resolution power of $g_n$, it suffices to consider the expansion of $f_{2}(u)$ in Chebyshev polynomials on $[-1,1]$.

In view of the bound (\ref{E:chebyshev_bound}), we now seek the maximum value of $f_{2}(u)$ along the Bernstein ellipse $e(\rho)$.  We first require the following lemma:

\begin{lemma}
For $\rho<E(T)$ and sufficiently large $\omega$, the maximum value of $|f_2(u)|$ on the Bernstein ellipse $e(\rho)$ occurs at $\theta = 0$.
\end{lemma}
\begin{proof}
Let $z=x+\ii y$.  Then 
\[
| \cos \omega z |^2 = \frac{1}{2} \left ( \cos 2 \omega x + \cosh 2 \omega y \right ).
\]
Hence, for sufficiently large $\omega$, the maximal value of $| \cos \omega z |$ on some curve $C$ in the complex plane occurs at the point $z \in C$ where $|\Im z|$ is maximized.

When $u = \frac{1}{2} (\rho^{-1} e^{-\ii \theta} + \rho e^{\ii \theta} )$, we may write $f_{2}(u) = \cos \omega T [X(\theta)+\ii Y(\theta) ]$, where $X(\theta)$ and $Y(\theta)$ are defined by
\[
\cos \left [ X(\theta)+\ii Y(\theta) \right ] = \frac{u}{2}(1-c(T)) + \frac{1}{2}(1+c(T)) = A(\theta)+\ii B(\theta),
\]
and, letting $\rho = e^{\eta}$ for $\eta > 0$ (recall that $\rho > 1$), 
\[
A(\theta) = \frac{1}{2}(1-c(T)) \cosh \eta \cos \theta + \frac{1}{2}(1+c(T)),\quad B(\theta) = \frac{1}{2}(1-c(T)) \sinh \eta \sin \theta.
\]
Expanding the cosine, and equating real and imaginary parts, we find that
\[
\cos X(\theta) \cosh Y(\theta)  = A(\theta),\quad \sin X(\theta) \sinh Y(\theta)  = B(\theta).
\]
We seek to maximize $|Y(\theta)|$.  Note trivially that $Y(\theta) $ cannot vanish identically for all $\theta$, and thus it suffices to consider only those $\theta$ for which $Y(\theta) > 0$.  Let $Z(\theta) = \cosh^2 Y(\theta)$.  Then $Z(\theta)$ is defined by
\[
\frac{A^2(\theta)}{Z(\theta)} + \frac{B^2(\theta)}{Z(\theta)-1} = 1.
\]
Rearranging,
\[
Z^2(\theta) - (1+A^2(\theta) + B^2(\theta) ) Z(\theta) + A^2(\theta) = 0.
\]
To complete the proof, it suffices to show that $Z$ attains its maximum value at $\theta = 0$.  Suppose not.  Then $Z'(\theta_0) = 0$ for some $\theta_0 \neq 0 $ (with $Z(\theta_0) >1$), and, after differentiating the above expression, we obtain 
\begin{equation*}
\left ( A'(\theta_0) A(\theta_0) + B'(\theta_0) B(\theta_0) \right )Z(\theta_0) = A'(\theta_0) A(\theta_0),
\end{equation*}
i.e.
\begin{align}
\sin \theta_0 \left ( (1-c(T) ) \right . & \cos \theta_0  \left .+ (1+c(T)) \cosh \eta \right )  Z(\theta_0)  \nonumber
\\
&= \sin \theta_0 \cosh \eta \left ( 1 + c(T) - (c(T)-1) \cosh \eta \cos \theta_0 \right ).\nonumber
\end{align}
Let us assume first that $\theta_0 \neq 0 , \pm \pi$.  Hence
\begin{align}
\label{E:Z}
\left ( (1-c(T) ) \right . & \cos \theta_0  \left .+ (1+c(T)) \cosh \eta \right )  Z(\theta_0)  \nonumber
\\
&= \cosh \eta \left ( 1 + c(T) - (c(T)-1) \cosh \eta \cos \theta_0 \right ).
\end{align}
Suppose the left-hand side vanishes of~\eqref{E:Z} vanishes, i.e.
\begin{equation}
\label{step1}
\cos \theta_0 = \frac{1+c(T)}{c(T)-1} \cosh \eta.
\end{equation}
Since $\sin(\theta_0) \neq 0$ the right-hand side of~\eqref{E:Z} must also vanish, and therefore
\begin{equation} \label{step2}
\cos \theta_0 = \frac{1+c(T)}{c(T)-1} \mathrm{sech} \eta.
\end{equation}
Equations~\eqref{step1} and~\eqref{step2} cannot hold simultaneously (since $\eta > 0$), and therefore the left-hand side of~\eqref{E:Z} cannot vanish.  We deduce that 
\[
Z(\theta_0) = \frac{ \cosh \eta \left ( 1 + c(T) - (c(T)-1) \cosh \eta \cos \theta_0 \right )}{(1-c(T) ) \cos \theta_0+ (1+c(T)) \cosh \eta}.
\]
We now substitute this into the quadratic equation for $Z(\theta)$ to give
\[
\left ( \cos \theta_ 0 - \cosh \eta \right )^2 \left ( (c(T)-1) \cosh \eta \cos \theta_0  - (1+c(T)) \right ) = 0.
\]
Since $\eta > 0$, we must have that
\[
\cos \theta_0 = \frac{1+c(T)}{(c(T)-1)} \mbox{sech} \eta.
\]
However, substituting this into~\eqref{E:Z}, gives $Z(\theta_0) = 0$, a contradiction.  Therefore $\theta_0 = 0,\pm \pi$.  It is easy to see that $\theta_0 = \pm \pi$ leads to a smaller value of $Z(\theta_0)$ than $\theta_0 = 0$.  Hence the result follows.
\end{proof}
Using this lemma, we deduce that the $n$th coefficient $a_n$ of the Chebyshev expansion of $f_2(u)$ is bounded by 
\begin{align*}
B(\omega,n,\rho,T) &= \frac{2}{\rho^n} \cos \left [ \omega T \cos^{-1} \left ( \frac14 \left(\rho+\frac{1}{\rho}\right)(1-c(T)) + \frac{1}{2}(1+c(T)) \right ) \right ]
\\
&=\frac{2}{\rho^n} \cos \left [ i \omega T \cosh^{-1} \left (\frac14 \left(\rho+\frac{1}{\rho}\right)(1-c(T)) + \frac{1}{2}(1+c(T)) \right ) \right ].
\end{align*}
We proceeded in \S\ref{s:OPSresolution} by computing the roots of the partial derivative of the bound with respect to $\rho$. The same approach applies here, but unfortunately it does not lead to explicit expressions. However, a simple modification makes the bound more manageable.  We write
\[
\cos (x) = \frac{1}{2} ( e^{i x} + e^{-i x} ).
\]
Since the argument of the cosine for $\rho>1$ is purely imaginary, with positive imaginary part, the second exponential dominates the first. Hence, for large $\omega$, we may approximate $B(w,n,\rho,T)$ by 
\begin{align}
 \tilde{B}(\omega,n&,\rho,T)  \nonumber
 \\
 &:= \frac{1}{\rho^n} \exp \left[  \omega T \cosh^{-1}\left(\frac14 \left(\rho+\frac{1}{\rho}\right) \left(1-c(T)\right) +\frac12 \left(1+c(T)\right) \right) \right].\label{E:fourext_bound_exp}
\end{align}
One can then explicitly find roots of the partial derivative of $\tilde{B}$ with respect to $\rho$.  We have
\begin{align}
& \rho^* (n)  \nonumber
 \\
 =& -\frac{n^2 (c(T)+3) + \omega^2 T^2 (c(T)-1) \pm 2n \sqrt{\omega^2 T^2 (c(T)^2-1)+2 n^2(c(T)+1)}}{\omega^2 T^2 (c(T)-1) + n^2 (1-c(T))}.\label{E:roots_exp}
\end{align}
These two roots are again each other's inverse.  One quickly verifies that the square root is real and positive if and only if 
\begin{equation}\label{E:n_exp}
  n \geq \frac12 \omega T \sqrt{2-2c(T)} = \frac{1}{2} \omega r(T),
\end{equation}
with
\begin{equation}\label{E:rT}
 r(T) := T \sqrt{2 - 2c(T)} = 2T \sin \left ( \frac{\pi}{2 T} \right).
\end{equation}
If the condition~\eqref{E:n_exp} is satisfied, selecting the $+$ sign in~\eqref{E:roots_exp} yields the root that is greater than $1$.

A little care is necessary when applying the bound~\eqref{E:fourext_bound_exp}.  Recall that $f_{2}(u)$ is only analytic in Bernstein ellipses $e(\rho)$ with $\rho<E(T)$.  It may be the case that $\rho^{*}(n) \geq E(T)$ and therefore we cannot use this bound directly.  However, explicit computation yields the condition
\begin{equation}\label{E:condition_rstar}
E(T) \geq \rho^*(n) \quad \Leftrightarrow \quad n \geq \frac{2 }{\sqrt{6+2 c(T)}} \omega T = \frac{1}{\sqrt{1+\cos^2 \left ( \frac{\pi}{2T} \right )}} \omega T.
\end{equation}
Moreover, if $r(T)$ is given by~\eqref{E:rT}, then 
\[
\frac{1}{2} r(T) < \frac{T}{\sqrt{1+\cos^2 \left ( \frac{\pi}{2T} \right )}} ,\quad \forall T > 1,
\]
With this to hand, we are now able to distinguish between the following cases:
\begin{enumerate}
\item $\mathbf{n < \frac12 \omega r(T)}.$ In this case, the argument of the square root in~\eqref{E:roots_exp} is negative and hence both roots are imaginary. The function $\tilde{B}(\omega,n,\rho,T)$ is either monotonically increasing or decreasing as a function of $\rho$ on $[1,\infty)$. Reasoning as before, note that
\begin{equation}\label{E:Btilde_derivative}
\frac{\partial \tilde B}{\partial \rho} \big |_{\rho=1} = \frac{1}{2} \omega r(T) - n.
\end{equation}
The partial derivative vanishes precisely when $n = \frac12 \omega r(T)$ and it is positive for smaller $n$. Thus, the bound is increasing and the best possible bound we can find of the form~\eqref{E:chebyshev_bound} is
\[
|a_n | \leq \tilde B(\omega,n,1,T) = 2,\quad n < \frac{1}{2} \omega r(T).
\]

\item $\mathbf{\frac{1}{2} \omega r(T) <n < \frac{T}{\sqrt{1+\cos^2 \left ( \frac{\pi}{2T} \right )}} \omega}$. In this case the argument of the square root in~\eqref{E:roots_exp} is positive. Moreover, the condition $n < \omega T$ ensures that the overall expression for both roots $\rho^*$ is positive (note that the denominator switches sign at $n=\omega T$). Thus, there is a minimum of $\tilde B(\omega,n,\rho,T)$ at $\rho=\rho^*(n) > 1$, and this minimum satisfies $\rho^*(n) < E(T)$. Therefore, we obtain the bound
\[
|a_{n+k} | \leq \frac{2 M_{\rho^*(n)}}{(\rho^*(n))^{n+k}} \leq \frac{1}{(\rho^*(n))^{k}} \tilde{B}(\omega,n,\rho^*(n),T) \leq \frac{2}{(\rho^*(n))^k},\quad \forall k \in \N.
\]
Hence, we deduce exponential decay of the coefficients $a_{n}$ once $n$ exceeds $n > \frac{1}{2} \omega r(T)$.
\item $\mathbf{\frac{T}{\sqrt{1+\cos^2 \left ( \frac{\pi}{2T} \right )}} \omega < n < \omega T}$.  The arguments of the previous case still hold, but now $\rho^*(n) \geq E(T)$.  Thus, the minimum value of $\tilde B(\omega,n,\rho,T)$ is obtained at $\rho = E(T)$ and therefore
\begin{equation} \label{E:anexp}
|a_n| \leq  \frac{2}{E(T)^n} \exp \left [ \omega T \cosh^{-1} \left ( 2 + c(T) \right ) \right ].
\end{equation}
Note that this bound is valid for all $n \in \N$, not just $n$ in the stated range.  However, it only gives an accurate portrait of the coefficient decay once $n$ exceeds $\frac{T}{\sqrt{1+\cos^2 \left ( \frac{\pi}{2T} \right )}} \omega$.

\item $\mathbf{n > \omega T}$. The denominator of~\eqref{E:roots_exp} vanishes at $n=\omega T$ and switches sign for larger $n$, so that $\rho^*(n)$ becomes negative. Based on~\eqref{E:Btilde_derivative}, we conclude that the bound is monotonically decreasing as a function of $\rho$ on $[1,\infty)$. Once more we are limited to choosing $\rho \leq E(T)$, and therefore the estimate~\eqref{E:anexp} is also applicable in this case.
\end{enumerate}
This derivation establishes the resolution constant $r(T) = 2T \sin \left ( \frac{\pi}{2T} \right )$, but only for even oscillations $f(x) = \cos \omega \pi x$.  To prove the complete result, we need to obtain an equivalent statement for the odd functions
\begin{equation}\label{E:sine}
f(x) = \sin \omega \pi x.
\end{equation}
We proceed along similar lines to the cosine case.  First, note that the continuous Fourier extension of~\eqref{E:sine} equates to the orthogonal polynomial expansion of
\[
f_{2}(u) = \frac{\sin \left [ \omega T \cos^{-1} \left ( \frac{u}{2}(1-c(T)) + \frac{1}{2}(1+c(T) \right) \right ]}{\sqrt{1-\left [ \frac{u}{2}(1-c(T)) + \frac{1}{2}(1+c(T)) \right ]^2}} ,\quad u \in [-1,1].
\]
It is not immediately apparent that this function is analytic at $u=1$.  However, simple analysis reveals that the square-root type singularity is removable, and consequently $f_2(u)$ is analytic in any Bernstein ellipse $e(\rho)$ with $\rho<E(T)$.

Proceeding as before, we find that the $n$th Chebyshev polynomial coefficient of $f_2(u)$ admits the bound $|a_n| \leq B(\omega,n,\rho,T)$, $\forall \rho > 1$, where
\[
B(\omega,n,\rho,T) = \frac{2 \sinh \left [ \omega T \cosh^{-1} \left ( \frac{1}{4}(\rho+\frac{1}{\rho})(1-c(T)) + \frac{1}{2}(1+c(T))\right ) \right ]}{\rho^n \sqrt{C(\rho,T)}},
\]
and $C(\rho,T) = \left [ \frac{1}{4}(\rho+\frac{1}{\rho})(1-c(T)) + \frac{1}{2}(1+c(T)) \right ]^2-1$.  Suppose first that $n<\frac{1}{2} \omega r(T)$.  Letting $\rho \rightarrow 1^+$,  we find that
\[
|a_n| \leq \lim_{\rho \rightarrow 1^+} B(\omega,n,\rho,T) = 2 \omega T.
\]
Next suppose that $\frac{1}{2} \omega r(T) < n < \frac{T}{\sqrt{1+\cos^2 \left ( \frac{\pi}{2T} \right )}} \omega$.  Then, for sufficiently large $\omega$, we can replace $B(\omega,n,\rho,T)$ by
\[
\frac{\tilde B(\omega,n,\rho,T)}{\sqrt{C(\rho,T)}},
\]
where $\tilde B(\omega,n,\rho,T)$ is the bound~\eqref{E:fourext_bound_exp} of the cosine case.  Hence, we deduce that
\begin{align*}
|a_{n+k} | &\leq \frac{2}{(\rho^*(n))^k \sqrt{C(\rho^*(n),T)}}
\\
& = \frac{2}{(\rho^*(n))^k} \sqrt{\frac{n^2-\omega^2 T^2}{2 (1+c(T)) n^2 \omega^2 T^2+ (c(T)^2-1) \omega^4 T^4}},\quad  \forall k \in \N.
\end{align*}
Note that the square root term is bounded since $n> \frac{1}{2} \omega r(T)$.  Indeed, it attains its maximal value at $n=\frac{1}{2} \omega r(T) + 1$, and hence is bounded in both $n$ and $\omega$.  This confirms exponential decay of the coefficients $a_{n}$ for $n>\frac{1}{2} \omega r(T)$.

In the third scenario, i.e. $n > \frac{T}{\sqrt{1+\cos^2 \left ( \frac{\pi}{2T} \right )}} \omega$, we use the value $\rho=E(T)$ to give
\[
|a_n| \leq \frac{2}{(E(T))^n} \frac{\exp \left [ \omega T \cosh^{-1} (2+c(T)) \right ]}{\sqrt{(2+c(T))^2-1}}.
\]
Note the similarity of this bound with~\eqref{E:anexp} for the cosine case.

To sum up, we have shown the following theorem:
\begin{theorem}\label{t:exact_res}
The resolution constant $r=r(T)$ of the continuous Fourier extension $g_n$ is precisely $2T \sin \left ( \frac{\pi}{2T} \right )$.  In particular, $r(T) \sim 2$ for $T \approx 1$ and $r(T) \sim \pi$ for $T \gg 1$.
\end{theorem}
The resolution constant $r(T)$ is illustrated as a function of $T$ in Fig.~\ref{fig:resolution_constant}. In Fig.~\ref{fig:exact_resolution} we confirm this theorem with several numerical examples.  As discussed in \S\ref{ss:convergence_numerical}, the result of the numerical computation of~\eqref{E:ls} may not coincide with the exact solution $g_n$.  We shall discuss the question of resolution power of the numerical solution, as opposed to the exact solution, in detail in \S \ref{ss:resolution_numerical}.  For the moment, so that we can illustrate Theorem~\ref{t:exact_res}, we have removed this issue by carrying out the computations in Fig.~\ref{fig:exact_resolution} in additional precision.

\begin{figure}[t]
\begin{center}
    \includegraphics[width=5cm]{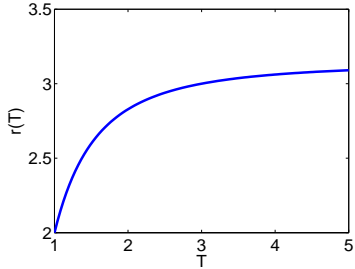}
 \caption{The resolution constant of Fourier extensions as a function of $T$.}\label{fig:resolution_constant}
\end{center}
\end{figure}

\begin{figure}[t]
\begin{center}
  \subfigure[$T=\sqrt{2}$]{
    \includegraphics[width=5cm]{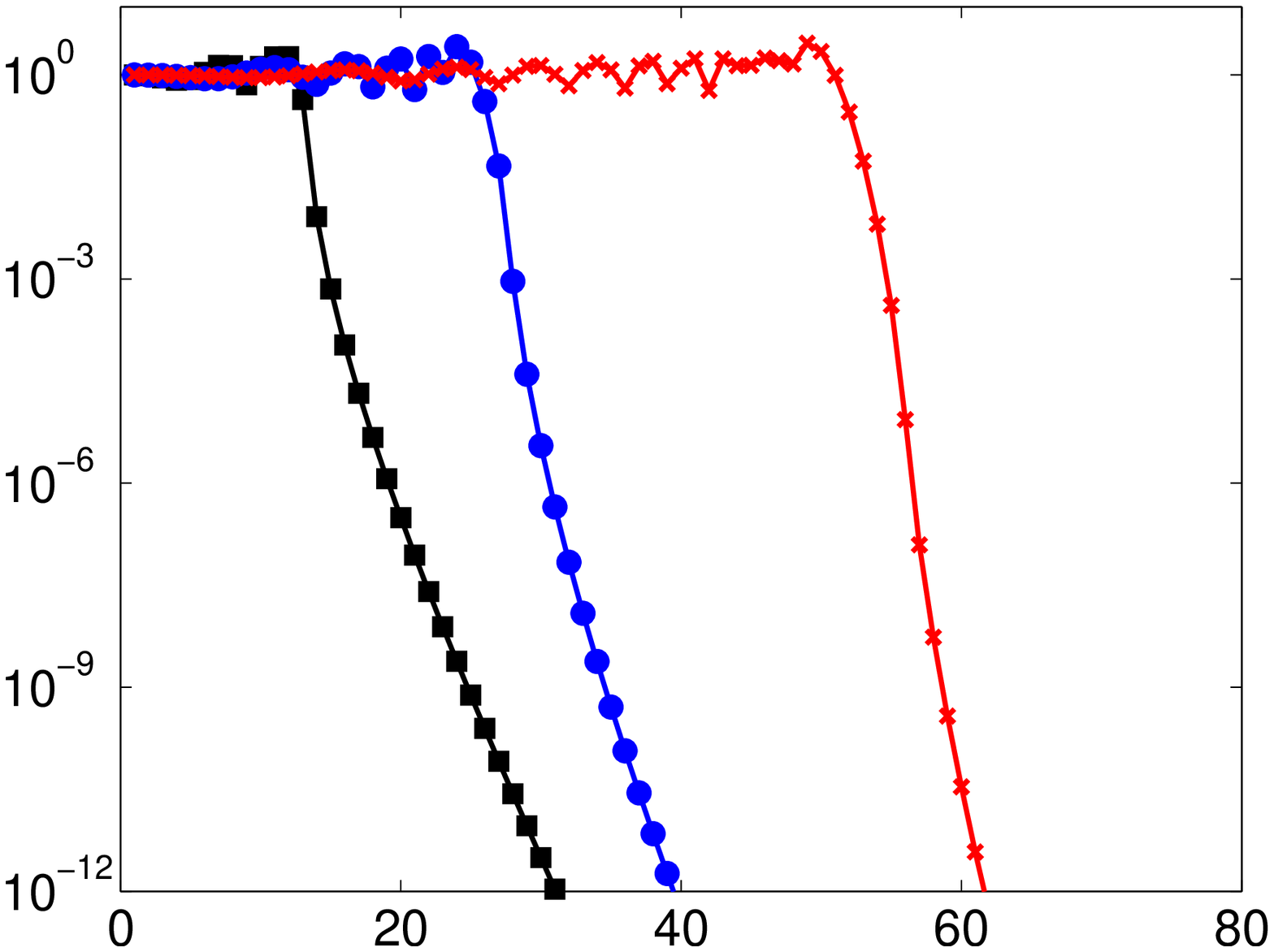}
  }
  \subfigure[$T=5$]{
    \includegraphics[width=5cm]{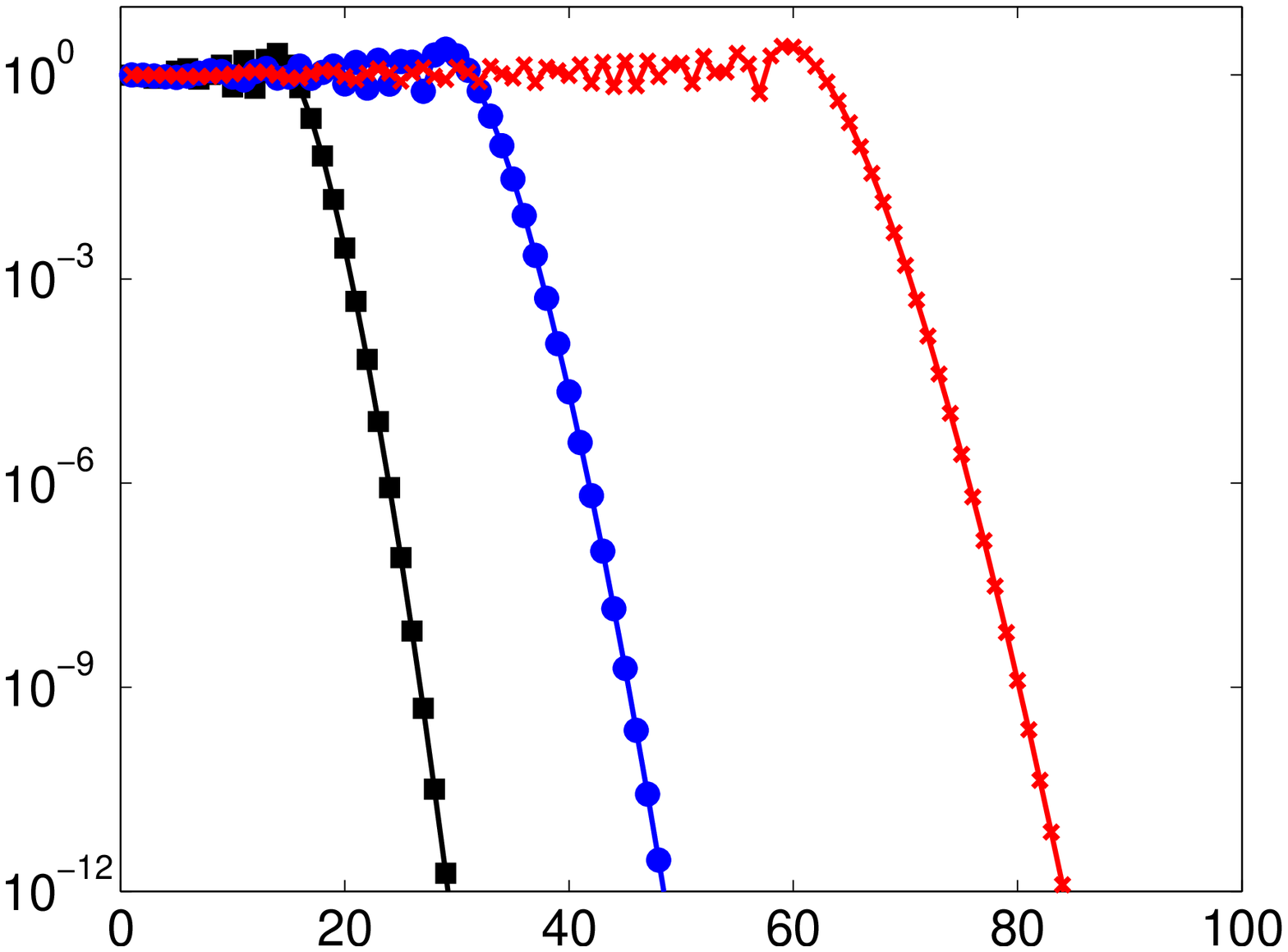}
  }
 \caption{The error $\| f - g_n \|_{L^\infty_{[-1,1]}}$, where $f(x) = \exp ( \ii \pi \omega x)$ and $\omega = 10,20,40$ (squares, circles and crosses respectively).}\label{fig:exact_resolution}
\end{center}
\end{figure}

As shown in the right panel of Fig.~\ref{fig:exact_resolution}, the resolution constant $r(T) \sim \pi$ for $T \gg 1$.  Incidentally, the fact that $r(T)$ is independent of $T$ for large $T$ can be seen by the studying the quantity $f_{2}(u)$.  For fixed $\omega$, we find that
\[
f_{2}(u) = \cos \left ( \frac{\pi w}{\sqrt{2}} \sqrt{1-u} \right ) + \BIGO{(T^{-2})},\quad T \rightarrow \infty.
\]
The leading term is independent of $T$, thus we expect $r(T)$ to approach a constant value as $T \rightarrow \infty$.  It is also natural to expect that this constant is the same as the resolution constant of polynomial approximations. Indeed, for large $T$ the functions $\CKPTX$ and $\SKPTX$ are not at all oscillatory on $[-1,1]$.  In particular, they are well approximated on $[-1,1]$ by their Taylor series. As such, the span of the first $n$ such functions is very close to the span of their Taylor series, which is exactly the space of polynomials of degree $n$.  Thus, for fixed $n$ and sufficiently large $T$, the Fourier extension $g_n$ of an arbitrary function $f$ closely resembles the best polynomial approximation to $f$ in the $L^2$ sense; in other words, the Legendre polynomial expansion.  The value of $\pi$ for the resolution constant arises from the discussion in \S\ref{s:OPSresolution}.

As commented at the start of this section, oscillations at `periodic' frequencies $\omega = \frac{m}{T}$, $m \in \Z$, are informative in that they illustrate when the Fourier extension behaves roughly like a classical Fourier series on $[-T,T]$, that is, when $T \approx 1$, and when it does not, i.e. $T \gg 1$.  The decay of the coefficients $a_n$ for these oscillations is also quite special.  Since $f(x) = \exp (\ii \frac{m}{T} \pi x )$ is entire and periodic on $[-T,T]$, the corresponding function $f_{2}(u)$ given by~\eqref{E:f2udef} is also entire in the variable $u$.  This means that $f_2(u)$ is analytic in any Bernstein ellipse $e(r)$, and hence the corresponding Chebyshev coefficients $a_n$ (respectively, the errors $\| f - g_n \|$) decay superexponentially fast, as opposed to merely exponentially fast at rate $E(T)$.  Comparison of the left and right panels of Fig.~\ref{fig:exact_resolution} illustrates this difference.  

We now summarize this observation, along with the other convergence characteristics derived above, in the following theorem:

\begin{theorem}\label{th:convergence}
The error $\| f - g_n \|$ in approximating $f(x) = \exp(\ii \pi \omega x)$ by its continuous Fourier extension $g_n$ is $\BIGO(1)$ for $n < \frac{1}{2} \omega r(T)$.  Once $n$ exceeds $\frac{1}{2} \omega r(T)$ the error begins to decay exponentially, and when $n > \frac{T}{\sqrt{1+\cos^2 \left ( \frac{\pi}{2T} \right )}} \omega $, the rate of exponential convergence is precisely $E(T)$, unless $\omega = \frac{m}{T}$ for some $m \in \Z$, in which case it is superexponential.
\end{theorem}

\subsection{Points per wavelength and resolution power of the discrete Fourier extension}
The discrete Fourier extension possesses exactly the same resolution power as its continuous counterpart.  This is a consequence of the fact that the error committed by the polynomial interpolant of a function at Chebyshev nodes can be bounded by a logarithmically growing factor (the Lebesgue constant) multiplied by the error of its expansion in Chebyshev polynomials \cite{rivlin1990chebyshev}.

However, the discrete Fourier extension does conveniently connect the concept of resolution power with the notion of points per wavelength (see \S\ref{s:introduction}).  As shown by Theorem~\ref{t:exact_res}, Fourier extensions require $r(T) = 2 T \sin \left ( \frac{\pi}{2T} \right )$ points per wavelength to resolve oscillatory behaviour, provided the corresponding nodes are of the form~\eqref{E:chebnodes}.

As it transpires, the exact value for the resolution constant can also be heuristically obtained by looking at the maximal spacing between the mapped symmetric Chebyshev nodes $x_i$ (as defined in~\eqref{E:chebnodes}).  Intuitively speaking, if the maximal node spacing for a set of $n$ nodes is $h$, then one can only expect to be able to resolve oscillations of frequency $\omega < \frac{1}{h}$.  For the mapped symmetric Chebyshev nodes, one can show that the maximal spacing
\[
h = x_{n} - x_{n-1} \sim \frac{1}{2n} r(T),\quad n \rightarrow \infty.
\]
Given that there are a total of $2n+2$ nodes, one also obtains the stipulated value for the resolution constant in this manner.  Note that identical arguments based on either equispaced or Chebyshev nodes in $[-1,1]$ give maximal grid spacings $h=\frac{2}{n}$ and $h=\frac{\pi}{n}$ respectively.  These correspond to the resolution constants of Fourier series and polynomial approximations, in agreement with previous discussions.

\subsection{Resolution power of numerical approximations}\label{ss:resolution_numerical}
As mentioned, it is not necessarily the case that the numerical solution of~\eqref{E:exactlinsys} (or its discrete counterpart) will coincide with the continuous (discrete) Fourier extension.  Therefore, the predictions of Theorem~\ref{t:exact_res} may not be witnessed in computations.  In Fig.~\ref{fig:num_resolution} we give results for the numerical computation of the continuous/discrete Fourier extensions.   For small $T$ there is little difference with the left panel of Fig.~\ref{fig:exact_resolution} (except, as mentioned in \S\ref{ss:resolution_numerical}, the computation of the continuous extension only attains around eight digits of accuracy).  However, when $T \gg 1$ the numerically computed extension appears to possess a much larger resolution constant than the value $r(T) \sim \pi$ predicted by Theorem~\ref{t:exact_res}.  In fact, as demonstrated in Fig.~\ref{fig:num_resolution2}, the numerical Fourier extension appears to have a resolution constant of $2T$ for all $T$, much like the na\"ive estimate 
given at the start of \S\ref{ss:resolution_exact}.  This observation is further corroborated in Fig.~\ref{fig:num_resolution3}, where a linear dependence of the numerical resolution constant with $T$ is witnessed.

\begin{figure}[t]
\begin{center}
 \subfigure[continuous, $T=\sqrt{2}$]{
    \includegraphics[width=5cm]{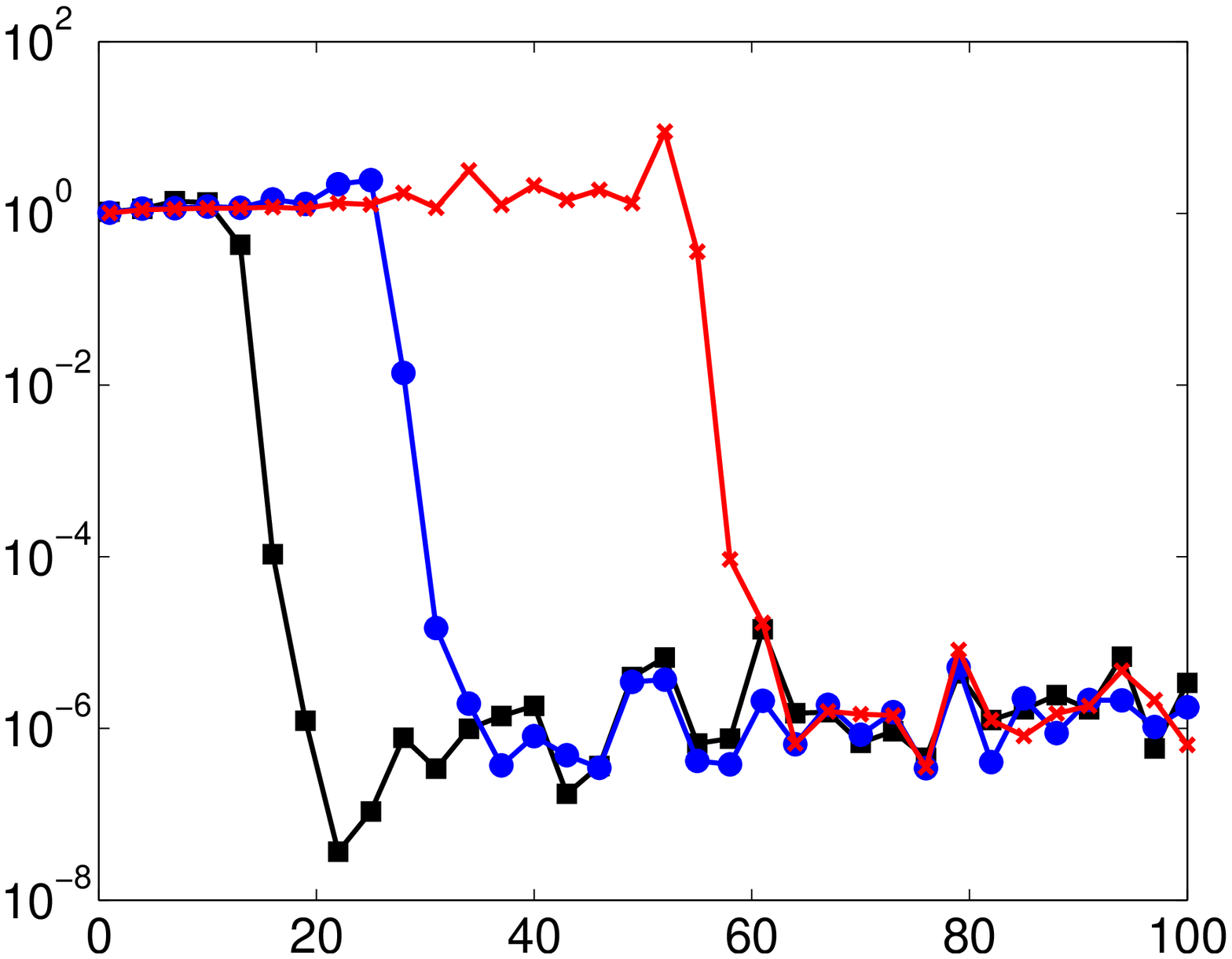}
  }
  \subfigure[continuous, $T=5$]{
    \includegraphics[width=5cm]{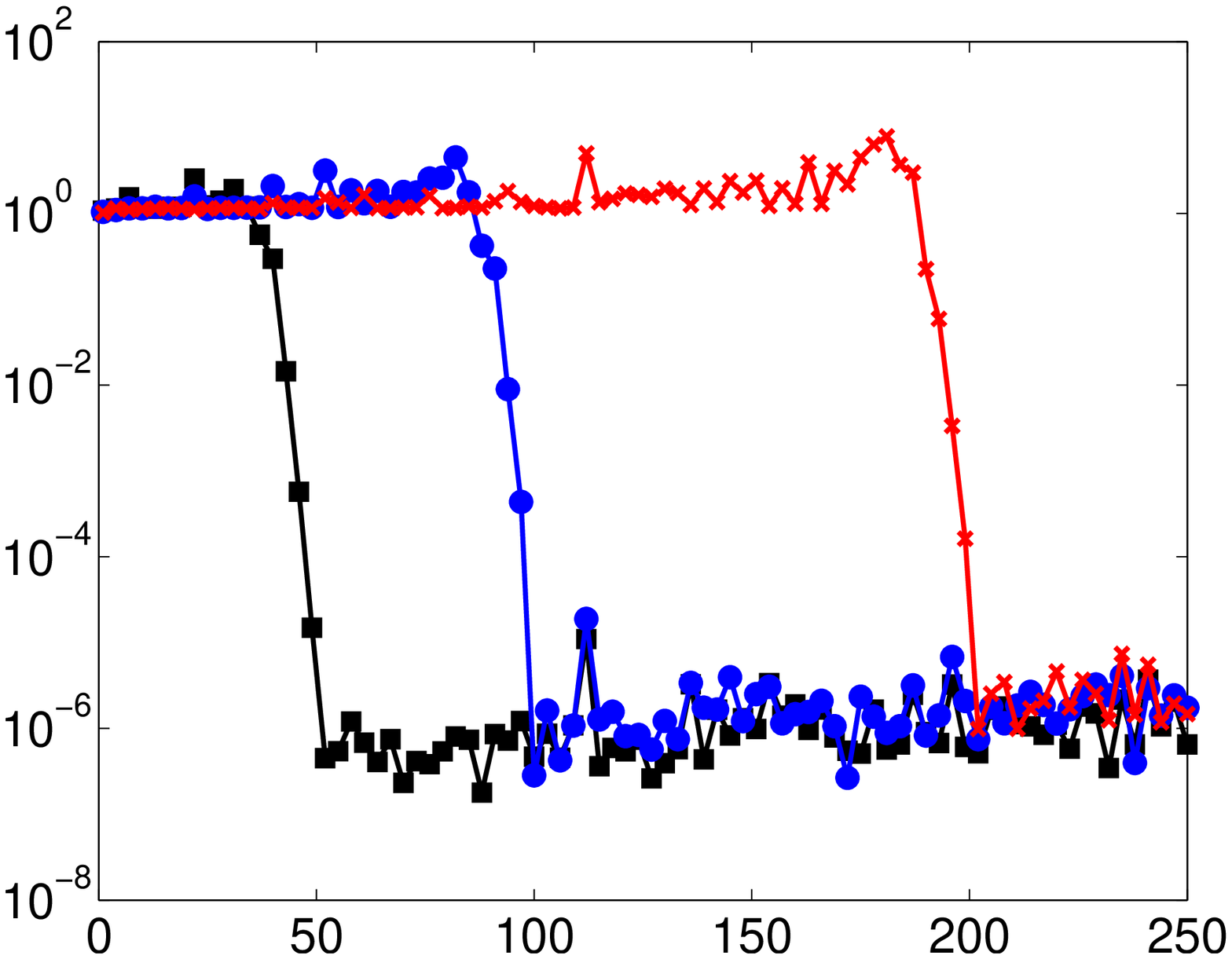}
  }
\\
  \subfigure[discrete, $T=\sqrt{2}$]{
    \includegraphics[width=5cm]{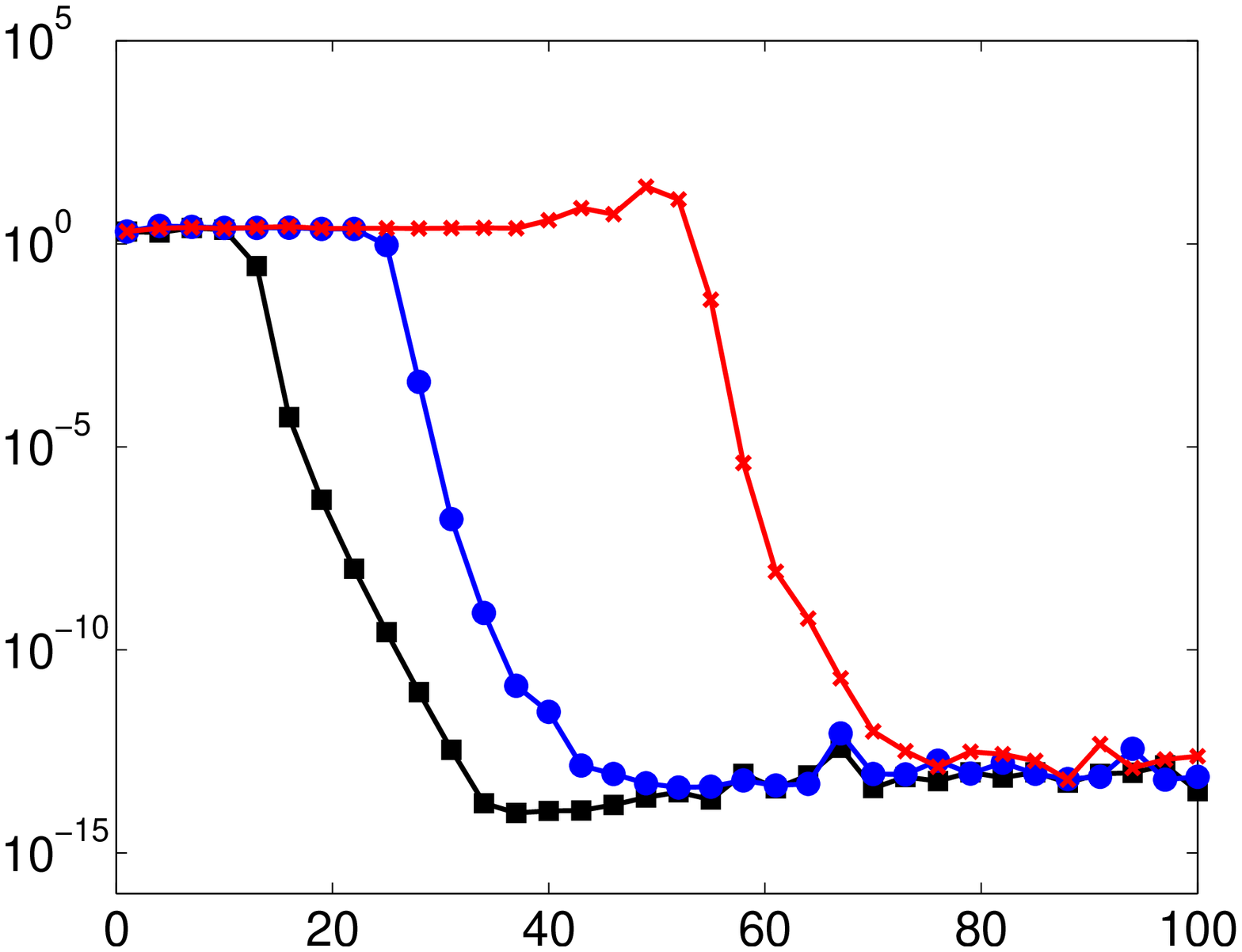}
  }
  \subfigure[discrete, $T=5$]{
    \includegraphics[width=5cm]{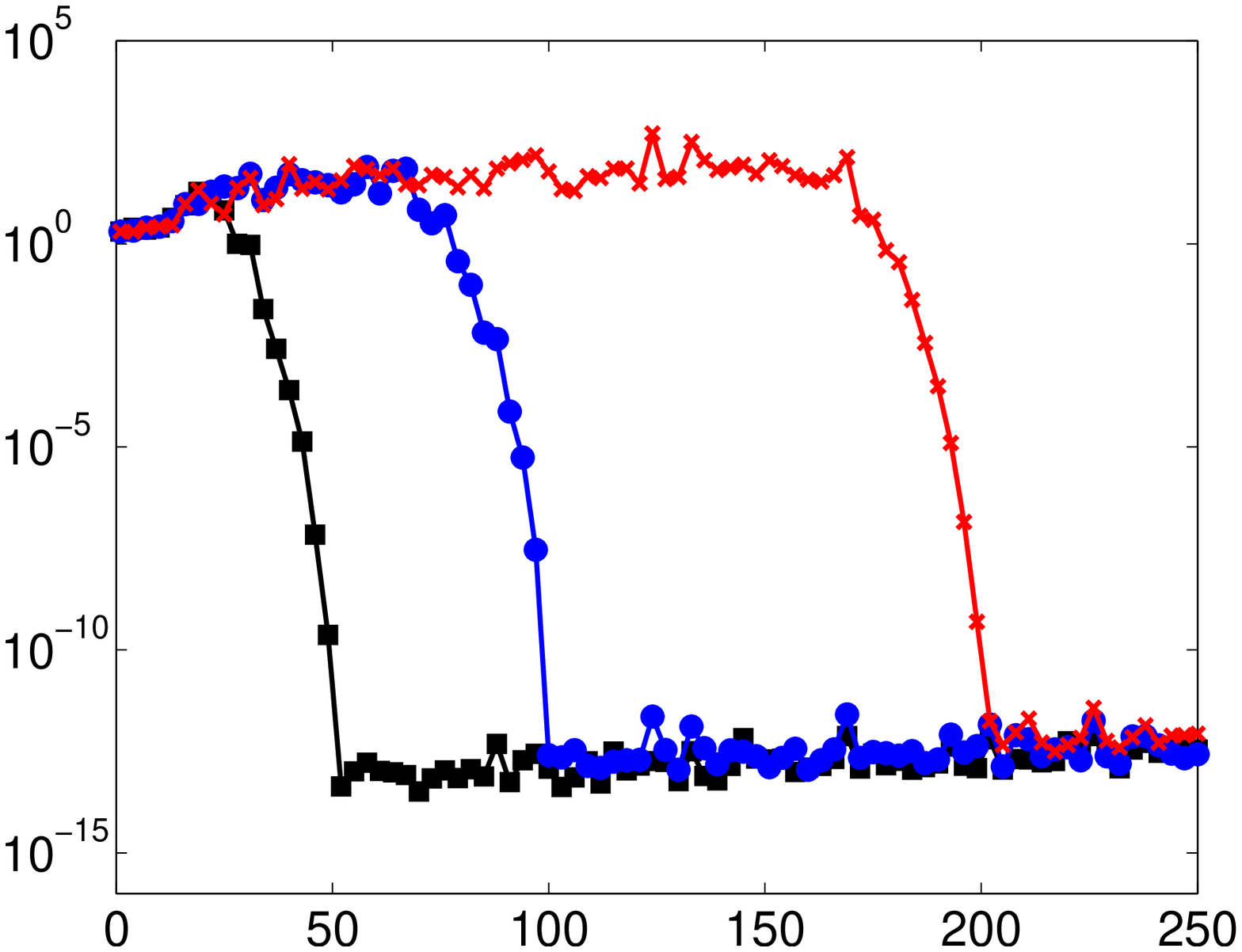}
  }
 \caption{The error $\| f - g_n \|_{L^\infty_{[-1,1]}}$, where $f(x) = \exp ( \ii \pi \omega x)$ and $\omega = 10,20,40$.}\label{fig:num_resolution}
\end{center}
\end{figure}

\begin{figure}[t]
\begin{center}
  \subfigure[continuous]{
    \includegraphics[width=5cm]{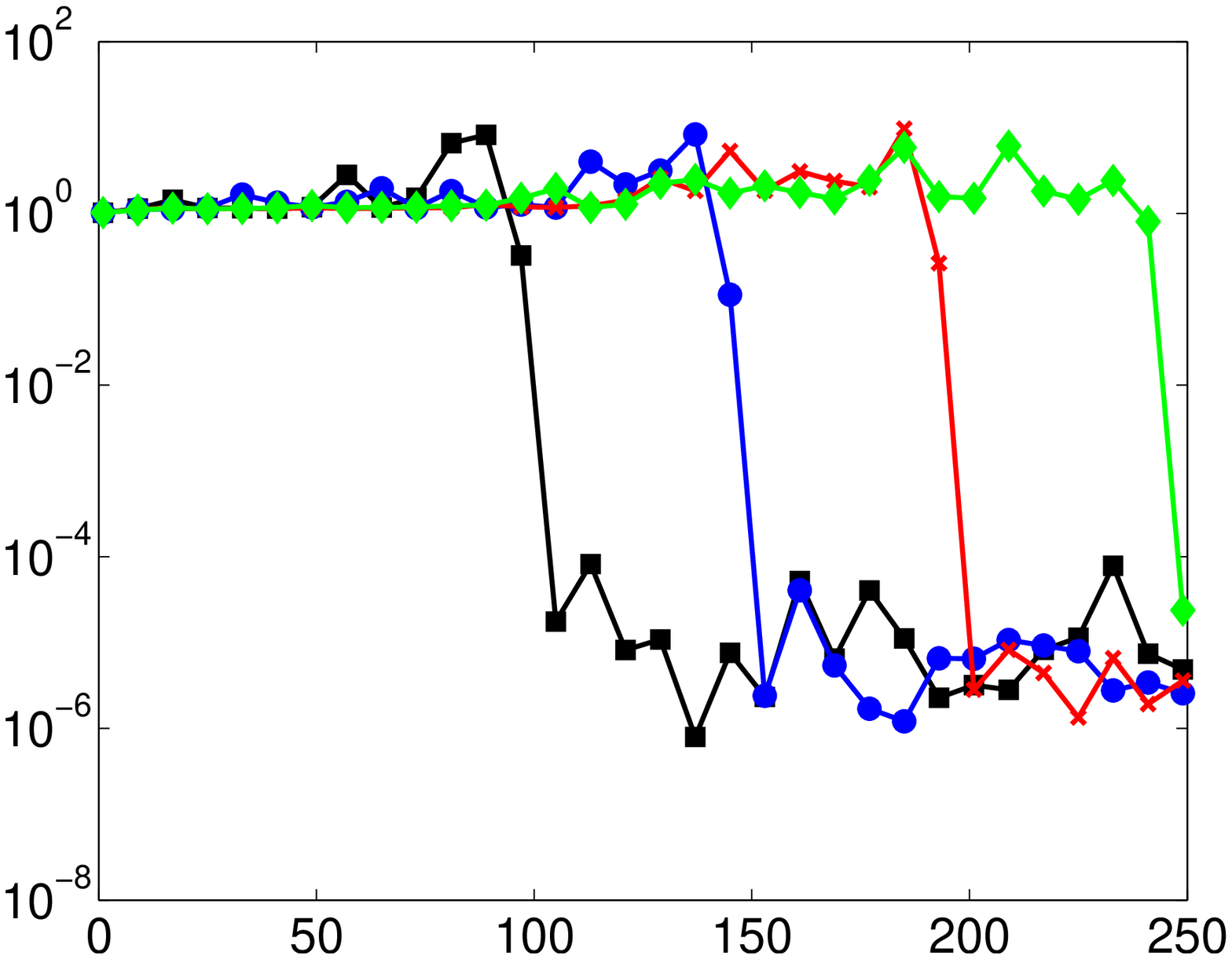}
  }
  \subfigure[discrete]{
    \includegraphics[width=5cm]{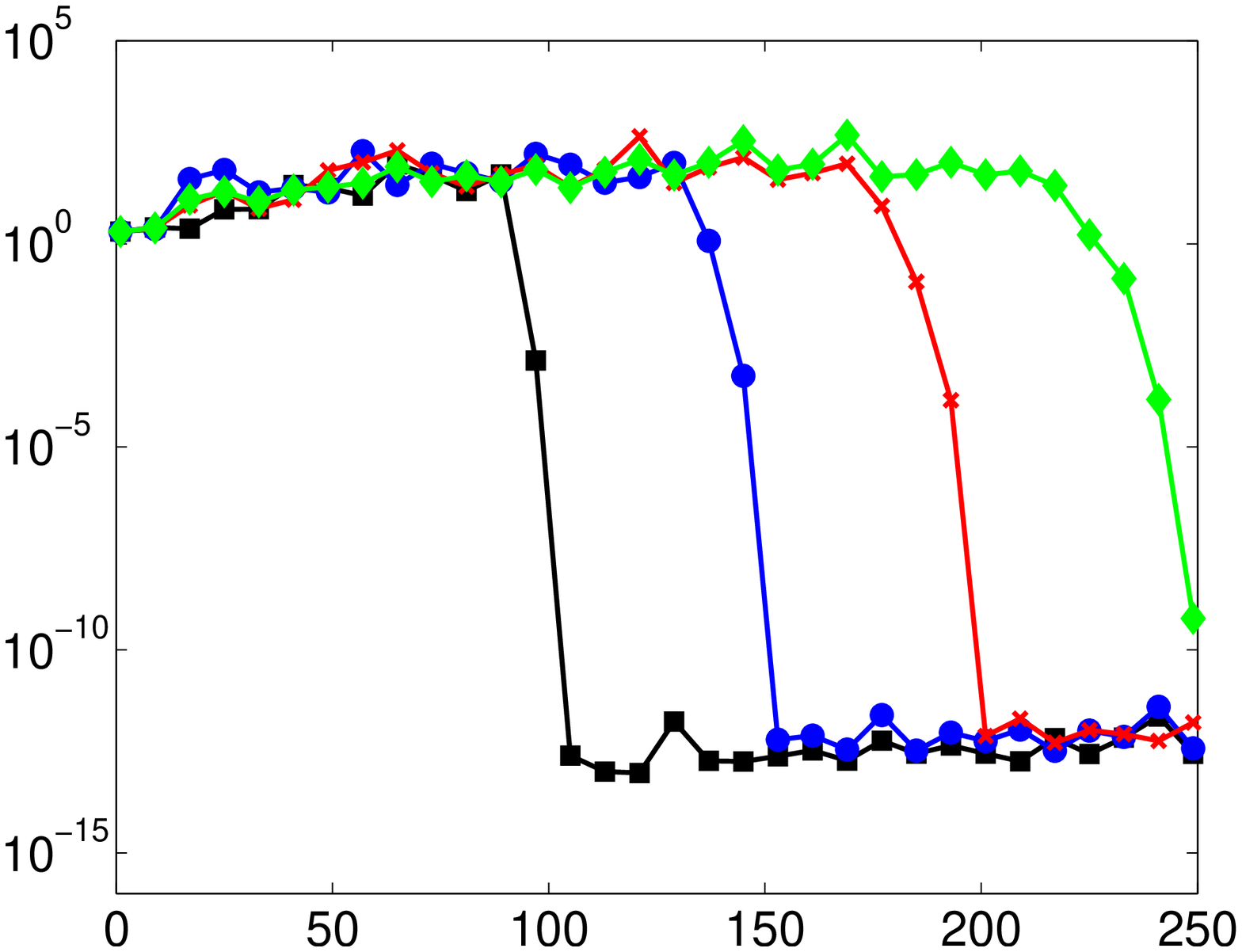}
  }
 \caption{The error $\| f - g_n \|_{L^\infty_{[-1,1]}}$, where $f(x) = \exp ( 50 \ii \pi x)$ and $T=2,3,4,5$ (squares, circles, crosses and diamonds respectively).}\label{fig:num_resolution2}
\end{center}
\end{figure}

\begin{figure}[t]
\begin{center}
    \includegraphics[width=5cm]{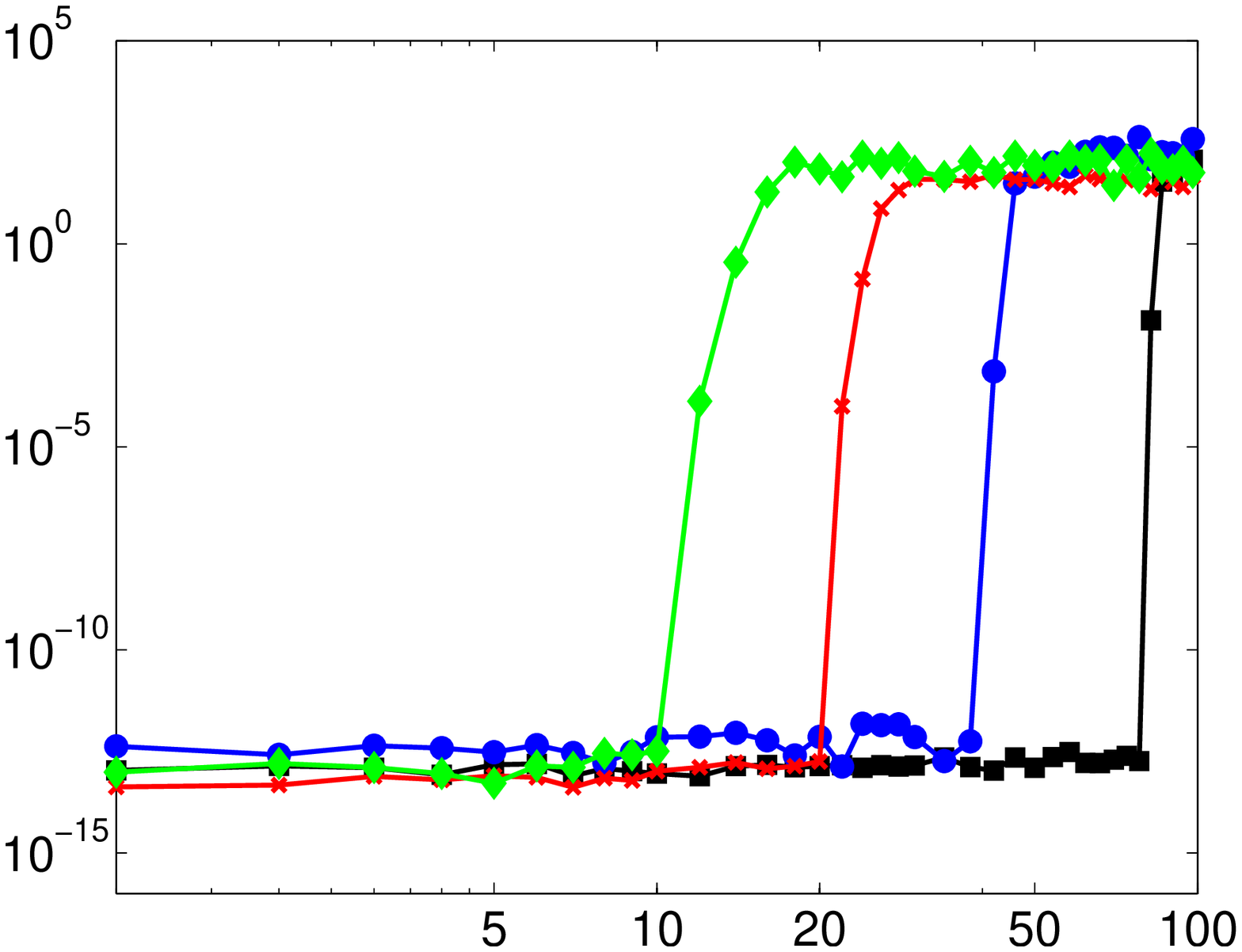}
 \caption{The error $\| f - g_{160} \|_{L^\infty_{[-1,1]}}$ against $\omega =1,2,\ldots,100$, where $g_n$ is the discrete Fourier extension of $f(x) = \exp (  \ii \omega \pi x)$ and $T=2,4,8,16$ (squares, circles, crosses and diamonds respectively).}\label{fig:num_resolution3}
\end{center}
\end{figure}

This difference can be explained intuitively. The improvement of Theorem~\ref{t:exact_res} over the na{\"i}ve estimate given at the beginning of \S\ref{ss:resolution_exact} revolves around the observation that slowly oscillatory functions in the frame can be recombined to approximate functions with higher frequencies with exponential accuracy. This is due to the redundancy of the frame. However, such combinations necessarily yield large coefficients: the orthogonal polynomials that give the exact solution grow rapidly outside $[-1,1]$ (see \cite{huybrechs2010fourier}) and, hence their coefficients when represented in the frame consisting of exponentials that are bounded on $[-T,T]$ must be large.  In fact, one has the following result:

\begin{lemma}\label{l:coeff_growth}
Let the continuous Fourier extension $g_n$ of $f(x) = \exp(i \pi \omega x)$ have coefficient vector $a \in \mathbb{C}^{2n+1}$.  Then
\begin{equation*}
\| \alpha \|_{l^2} \leq c(T) E(T)^{n},\quad n < \frac{1}{2} r(T) \omega,
\end{equation*}
for some constant $c(T) > 0$ depending on $T$ only.
\end{lemma}
\begin{proof}
By Parseval's relation, $\|  \alpha \|_{l^2} = \| g_n \|_{L^2_{[-T,T]}}$.  Writing $g_n$ as in Theorem \ref{th:exact_solution}, we obtain
\begin{equation*}
\|  \alpha \|_{l^2} = \| g_n \|_{L^2_{[-T,T]}} \leq\sum^{n}_{k=0} |a_k| \| T^{T}_{k} \|_{L^{\infty}_{[-1,1]}} +  \sum^{n-1}_{k=0} |b_k| \| U^{T}_{k} \|_{L^{\infty}_{[-1,1]}},
\end{equation*}
where $a = (a_0,\ldots,a_n)^{\top}$, $b = (b_0,\ldots,b_{n-1})^{\top}$ and $a_k$ and $b_k$ are given by~\eqref{E:ak_x} and~\eqref{E:bk_x} respectively.  It can be shown that
\begin{equation*}
\| T^{T}_{k} \|_{L^{\infty}_{[-1,1]}} , \| U^{T}_{k} \|_{L^{\infty}_{[-1,1]}} \sim E(T)^k,\quad k \rightarrow \infty,
\end{equation*}
(see \cite[Thm. 3.16]{huybrechs2010fourier} for the case $T=2$ -- the extension to general $T$ is straightforward), and therefore
\begin{equation*}
\|  \alpha \|_{l^2} \leq c \left ( \sum^{n}_{k=0} |a_k| E(T)^k + \sum^{n-1}_{k=0} | b_k | E(T)^k \right ),
\end{equation*}
for some $c>0$ independent of $n$ and $\omega$.  By the analysis of the previous section, $|a_k|, |b_k| \sim 1$ for $k < \frac{1}{2} r(T) \omega$, and hence we deduce the result.
\end{proof}
This lemma indicates that the coefficient vector of the continuous Fourier extension may well be exponentially large in the unresolved regime.  As discussed (see \S \ref{ss:convergence_numerical}), the numerical solution favours representations in the frame with bounded coefficients, since for large $n$ the system becomes underdetermined and most least squares solvers will seek a solution vector with minimal norm.  We therefore conclude that, whilst the resolution power of the frame is bounded by $\pi$, the resolution power of all representations in the frame with `reasonably small' coefficients is only $2T$. 

In Fig.~\ref{fig:num_thy} we illustrate this discrepancy by comparing the `theoretical' Fourier extension and its counterpart computed in standard precision.  In this example, $T=4$, i.e. $E(T) \approx 25$, and we therefore witness rapid exponential growth of the theoretical coefficient vector, in agreement with Lemma \ref{l:coeff_growth}.  In particular, the coefficient vector is of magnitude $10^{55}$ at the point at which the function $f$ begins to be resolved.  On the other hand, since for large $n$ the numerical solver favours small norm coefficient vectors, the computed coefficient vector initially grows exponentially and then levels off around $\| a \| \approx 10^{15}$.  Thus, in finite precision one cannot obtain the coefficient vector required to resolve $f$ at the point $n = \frac{1}{2} r(T) \omega$.  As commented previously, these arguments can be made rigorous by performing an analysis of the numerical Fourier extension \cite{BADHJMVFEStability}.

\begin{figure}[t]
\begin{center}
 \subfigure[Numerical extension]{
    \includegraphics[width=5cm]{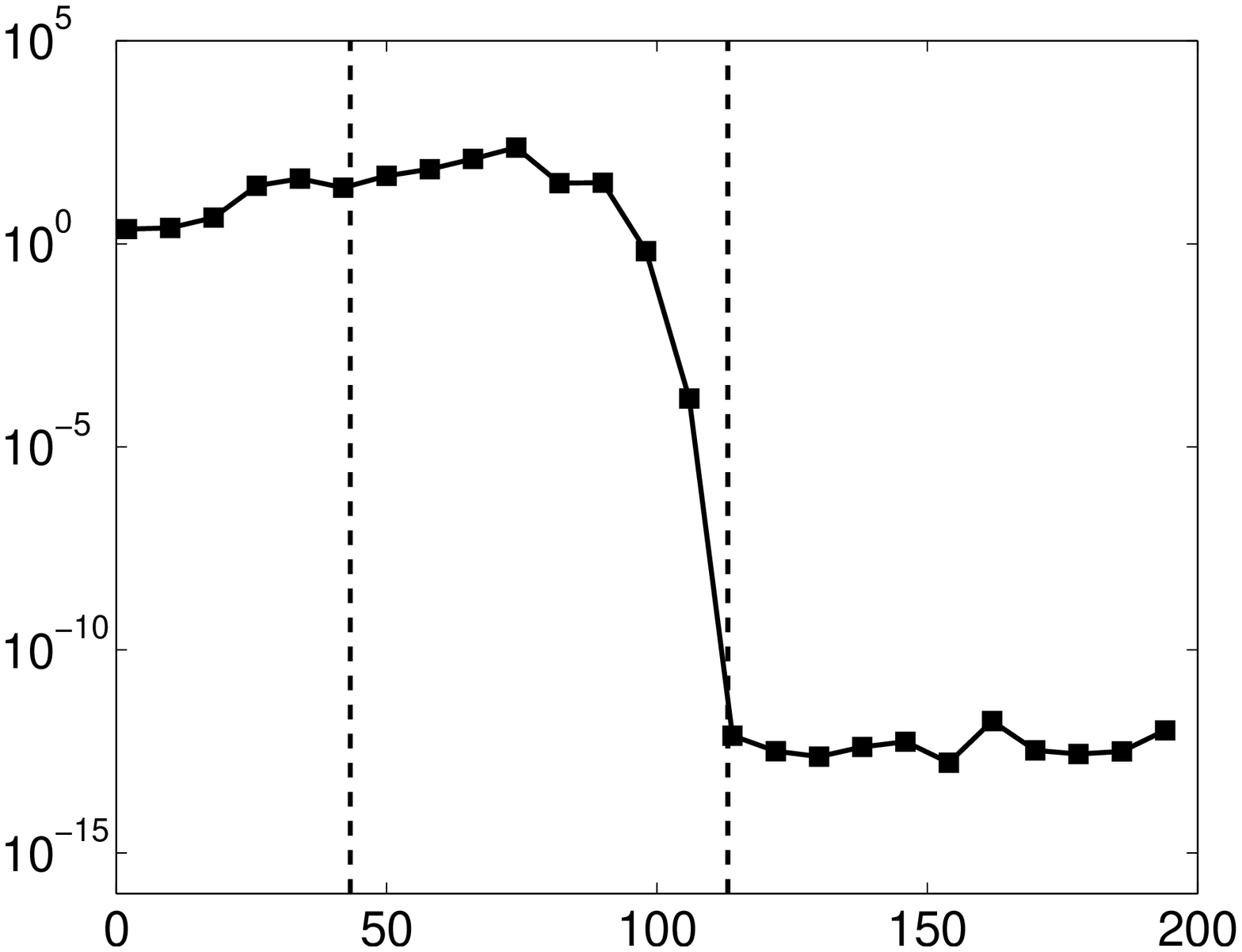}
  }
  \subfigure[Theoretical extension]{
    \includegraphics[width=5cm]{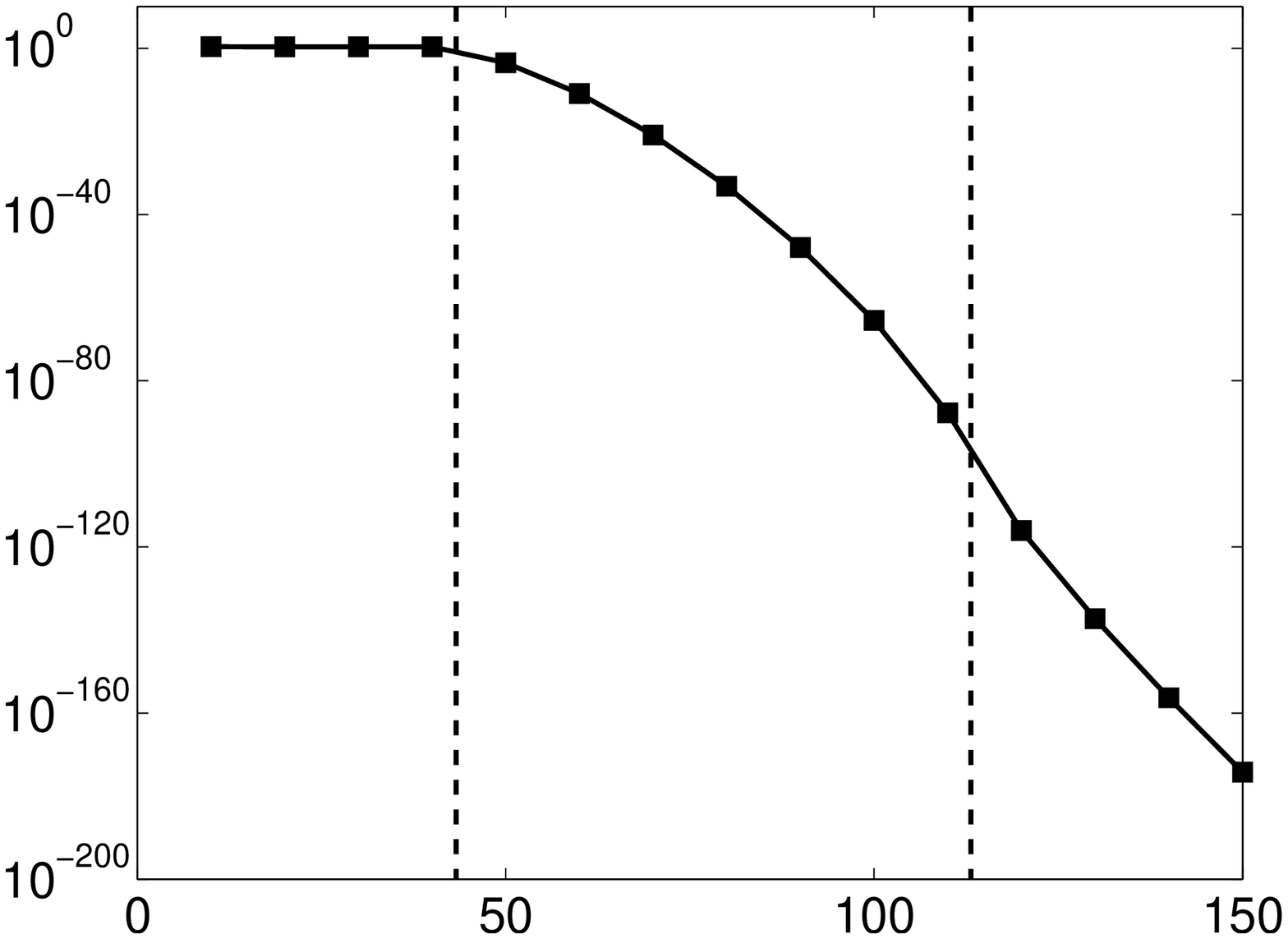}
  }
\\
  \subfigure[Numerical extension]{
    \includegraphics[width=5cm]{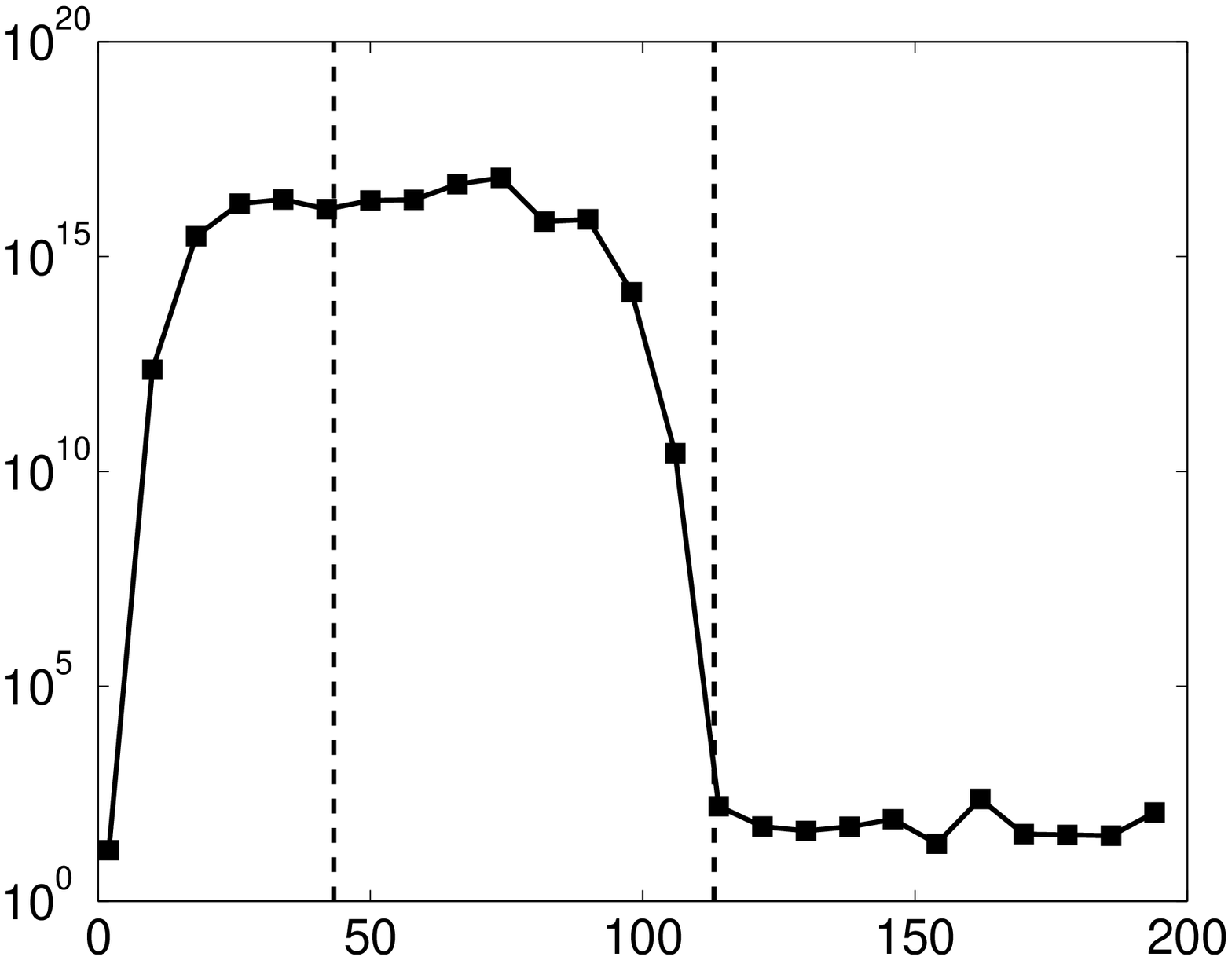}
  }
  \subfigure[Theoretical extension]{
    \includegraphics[width=5cm]{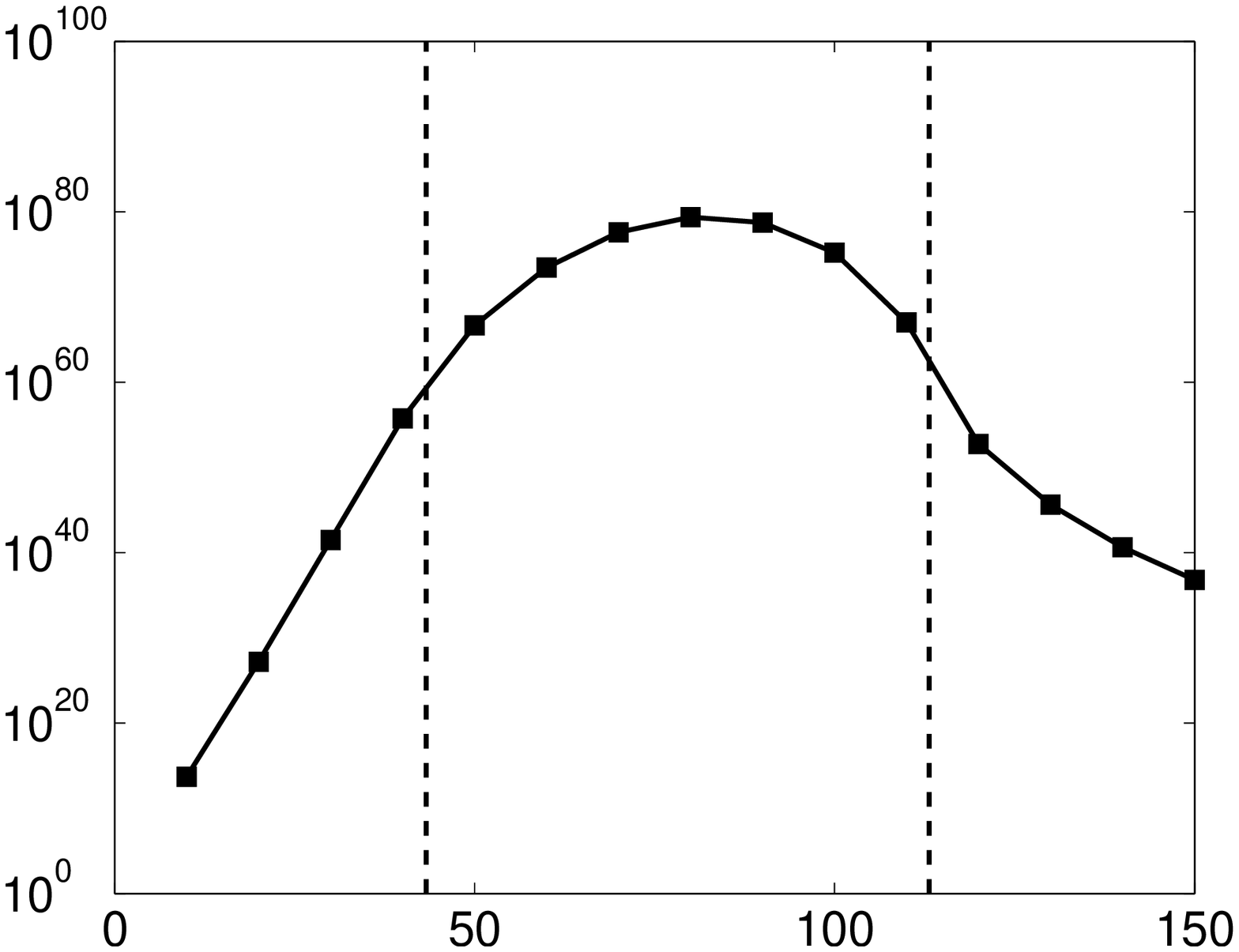}
  }
 \caption{The top row shows the error $\| f - g_n \|_{L^{\infty}_{[-1,1]}}$ where $f(x) = \exp(i \omega \pi x)$, $\omega =20 \sqrt{2}$, and $T=4$.  The bottom row gives the norm of the coefficient vector. The numerical extension was computed in Matlab, the theoretical extension was computed in Mathematica using higher precision arithmetic. The dotted lines indicate the values $n=\frac12 r(T) \omega$ and $n=\omega T$.}\label{fig:num_thy}
\end{center}
\end{figure}

\subsection{Fixed $T$ versus varying $T$}
Whilst the principal issue highlighted in the previous section, namely, that for large $T$ we witness $r(T) \approx 2 T$ rather than $r(T) \approx \pi$ in practice, is unfortunate, it is mainly the case $T \approx 1$ that is of interest, since this gives the highest resolution power.

However, with $T$ fixed independently of $n$, the resolution constant $r(T) > 2$; in other words, greater than the optimal Fourier resolution constant $r=2$ (although this difference can be made arbitrarily small).  One way to formally obtain the value of $2$ is to allow $T \rightarrow 1^+$ with increasing $n$.  For example, if
\begin{equation}\label{T_formal}
T = 1 + \frac{c}{n^\alpha},
\end{equation}
for some $c>0$ and $0 < \alpha < 1$, then we find that
\[
r(T) = 2 + \frac{2 c}{n^\alpha} + \BIGO(n^{-2 \alpha}) \rightarrow 1,\quad n \rightarrow \infty,
\]
thereby giving formally optimal resolution power.  The disadvantage of such a choice of $T$ is that one forfeits exponential convergence.  In fact,
\[
E\left(1 + \frac{c}{n^\alpha} \right) = 1 + \frac{c \pi}{n^a} + \BIGO(n^{-2 \alpha}),
\]
and therefore
\begin{equation}\label{E:varying_T_decay}
E\left(1 + \frac{c}{n^\alpha} \right)^{-n} \sim \exp(-c \pi  n^{1-\alpha} ),\quad n \rightarrow \infty,
\end{equation}
which indicates subexponential decay of the error (yet still spectral).  Larger values of $\alpha$ lead to slow, algebraic convergence, and are consequently inadvisable.

An alternative way to choose $T$ is found in the literature on the Kosloff--Tal--Ezer mapping \cite{boyd,KTEmapped} (recall the discussion in \S \ref{ss:relation}).  In \cite{KTEmapped} the authors suggest choosing the mapping parameter based on some tolerance $\epsilon_{\mathrm{tol}}$.  This limits the best achievable accuracy to $\mathcal{O}(\epsilon_{\mathrm{tol}})$ but gives a formally optimal time-step restriction. We may do the same with the Fourier extension, by solving the equation
\[
E(T)^{-n} = \epsilon_{\mathrm{tol}},
\]
This gives
\begin{equation}\label{eq:optimal_T}
T = T(n,\epsilon_{\mathrm{tol}}) = \frac{\pi}{4} \left ( \arctan \left ( \epsilon_{\mathrm{tol}} \right )^{\frac{1}{2n}} \right )^{-1}.
\end{equation}
Note that
\[
T(n,\epsilon_{\mathrm{tol}}) \sim 1 - \frac{\log (\epsilon_{\mathrm{tol}})}{\pi n} + \mathcal{O}(n^{-2}),\quad n \rightarrow \infty,
\]
and therefore
\[
r(T(n,\epsilon_{\mathrm{tol}})) \sim 2 - \frac{2\log (\epsilon_{\mathrm{tol}})}{\pi n} + \mathcal{O}(n^{-2}),\quad n \rightarrow \infty.
\]
Hence we expect formally optimal result power with this approach, as well as a best achievable accuracy on the order of $\epsilon_{\mathrm{tol}}$. 

As we show in the next section, choosing $T$ in this manner works particularly well in practice.

\subsection{Numerical experiments}\label{s:numexp}

In this final section we present numerical results for the Fourier extension applied to the oscillatory functions
\begin{align}
\label{E:f12def}
f_1(x) = (1+x^2) \cos 10 x \cos 100 \pi x,\quad f_2(x) = \mbox{Ai}(-36x-32),
\end{align}
where $\mbox{Ai}(z)$ is the Airy function.  Graphs of these functions are given in the top row of Fig.~\ref{fig:num_exp}.

\begin{figure}[t]
\begin{center}
 \subfigure[The function $f_1$]{
    \includegraphics[width=5cm]{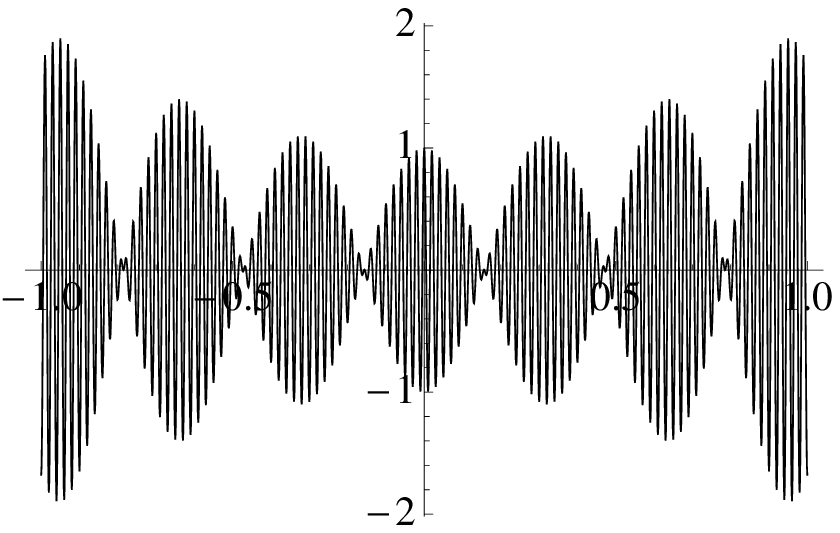}
  }
  \subfigure[The function $f_2$]{
    \includegraphics[width=5cm]{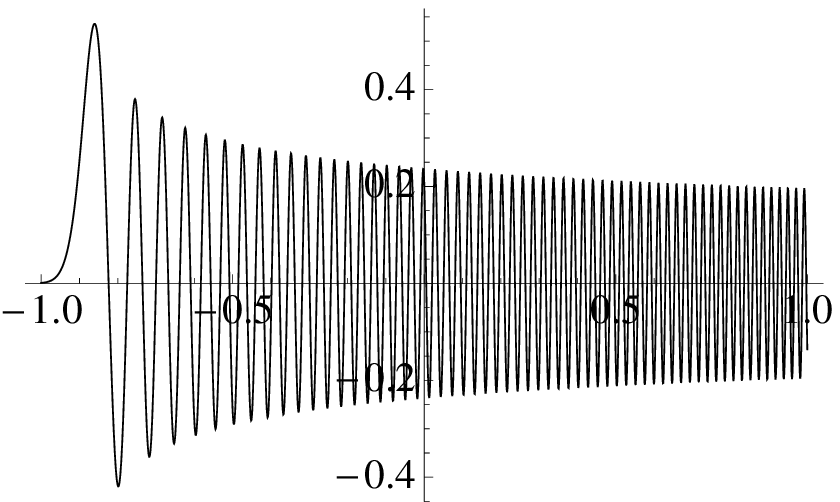}
  }
\\
  \subfigure[Fourier extension of $f_1$]{
    \includegraphics[width=5cm]{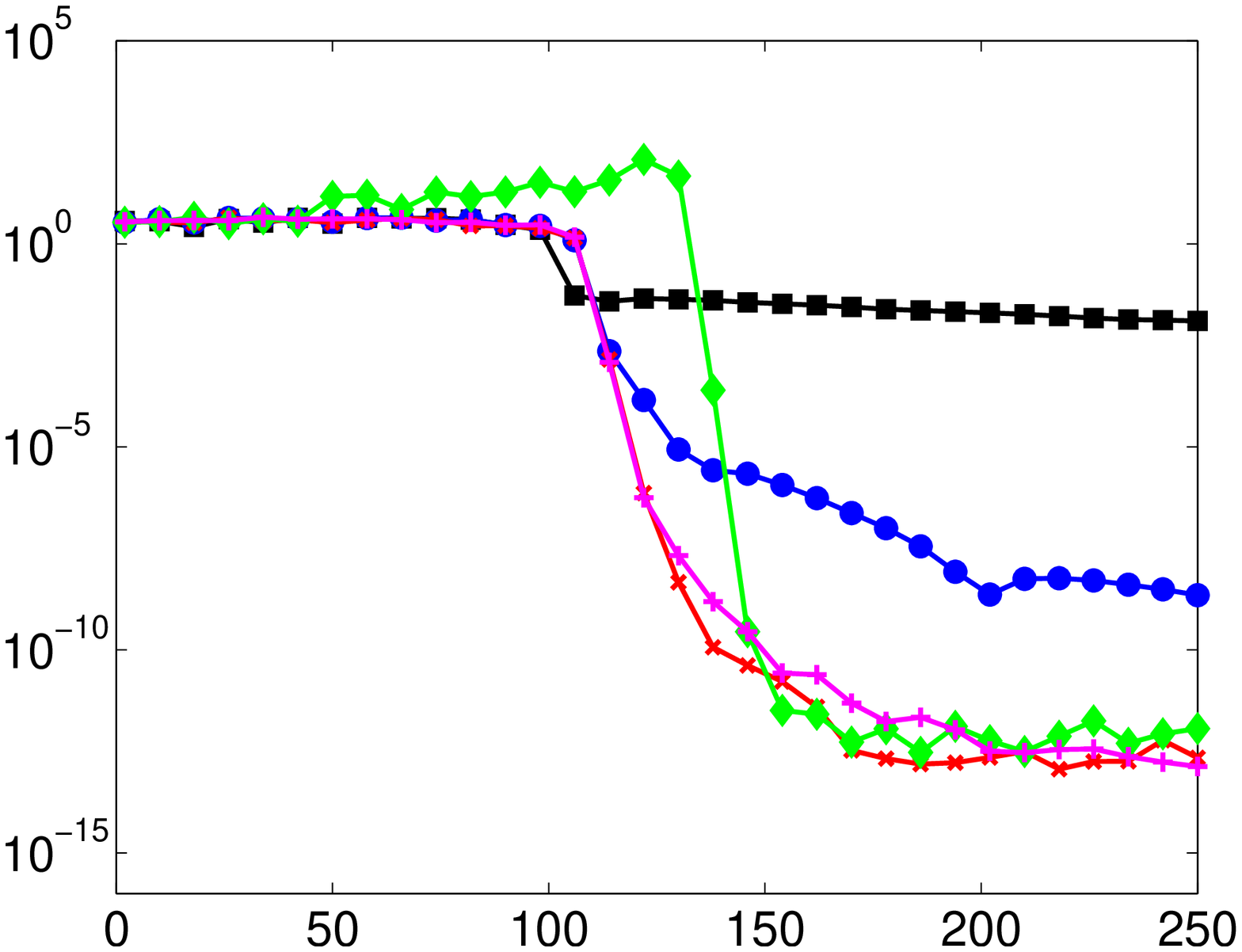}
  }
  \subfigure[Fourier extension of $f_2$]{
    \includegraphics[width=5cm]{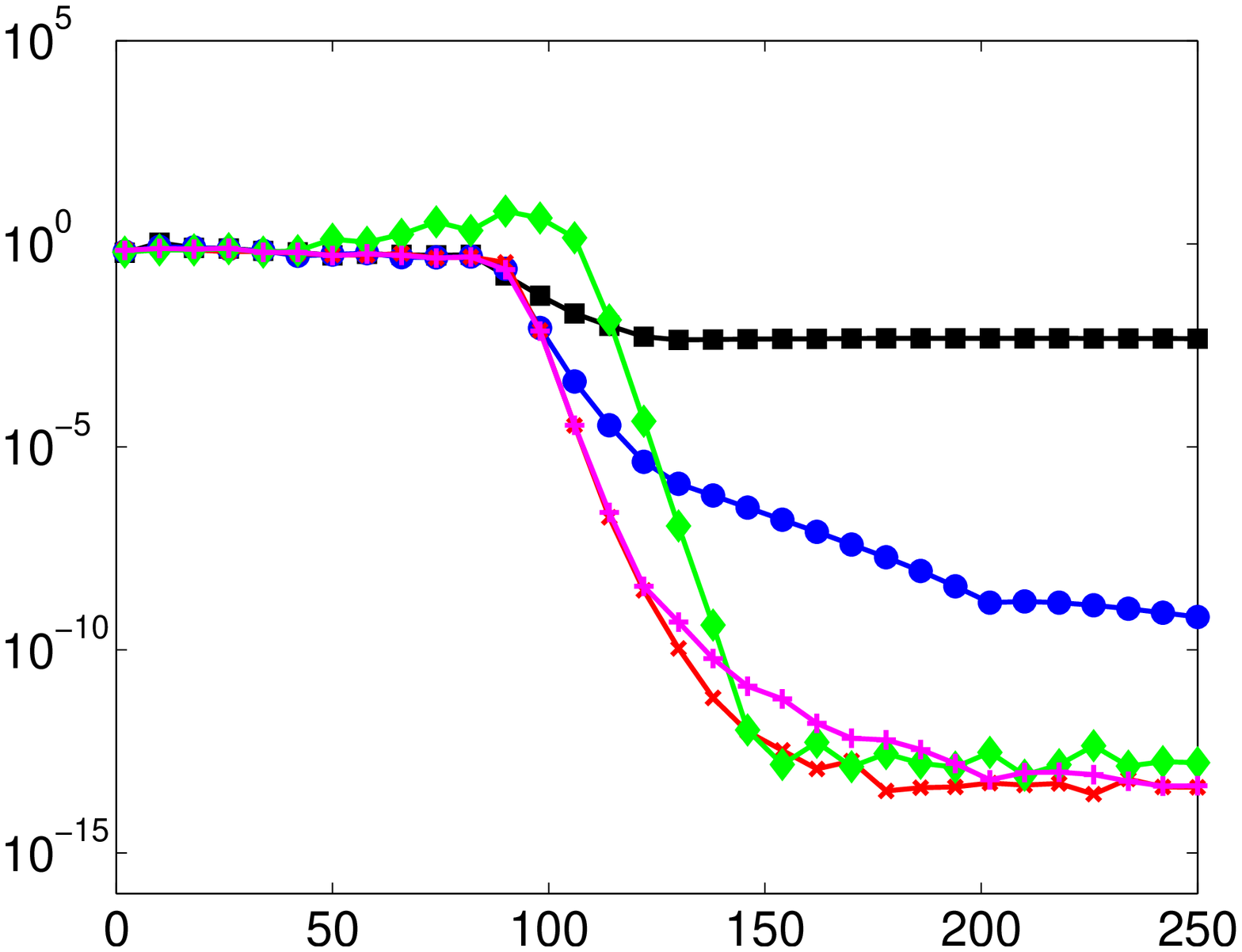}
  }
 \caption{The top row shows the functions $f_1$ and $f_2$ defined by~\eqref{E:f12def}.  The bottom row gives the error $\| f - g_n \|_{L^{\infty}_{[-1,1]}}$, where $T=1+\frac{1}{n}$ (squares), $T=1+n^{-\frac{2}{3}}$ (circles), $T=1+n^{-\frac{1}{2}}$ (crosses), $T=\frac{4}{3}$ (diamonds) and a formally optimal $T$ given by \eqref{eq:optimal_T} with $\epsilon_{\mathrm{tol}}=$1e-13 (pluses).}\label{fig:num_exp}
\end{center}
\end{figure}

In the bottom row of Fig.~\ref{fig:num_exp} we present the error committed by the discrete Fourier extension for various choices of $T$.  The convergence rates predicted in the previous section are confirmed for these examples.  For $T$ decreasing with $n$ the oscillations are resolved slightly sooner, but there is slower convergence in the resolved regime. Whether the approximation error is better or worse than that obtained from a fixed value of $T$ (in this case $T=\frac{4}{3}$) depends on what level of accuracy is desired. Yet, even when high accuracy is required, $\alpha=1/2$ (i.e. $T=1+\frac{1}{n^{1/2}}$) appears to be a good choice for these two examples, in spite of the theoretical subexponential convergence rate. The varying value of $T = T(n;\epsilon_{\mathrm{tol}})$, with tolerance set to $\epsilon_{\mathrm{tol}}=$1e-13, also yields comparable results. The convergence rate for the case of larger $\alpha=2/3$ is slower, as expected from the estimate~\eqref{E:varying_T_decay}.

As discussed, one motivation for using Fourier extensions is that, as rigorously proved in this paper, they offer a higher resolution power for small $T$ than polynomial approximations, for which the resolution constant is $\pi$ (see \S \ref{s:OPSresolution}).  In Fig.~\ref{fig:num_comp} we compare the behaviour of the Chebyshev expansion and the Fourier extension approximation for the functions~\eqref{E:f12def}.  Note that the latter resolves the oscillatory behaviour using fewer degrees of freedom, in agreement with the result of \S \ref{s:resolution}.  Indeed, for the first example function $f_1$ shown in the left panel, the Chebyshev expansion only begins to converge once $n$ exceeds $150$ (here, for purposes of comparison, the number of expansion coefficients is $2n$), in agreement with a resolution constant of $\pi$.

The results for the second function $f_2$ are shown in the right panel of Fig.~\ref{fig:num_comp}. The differences in resolution are smaller in this case. Yet, the fixed choice $T=8/7$ and varying choices of $T$ still require significantly fewer degrees of freedom than the Chebyshev expansion to resolve the oscillatory behaviour.  As before, a slower asymptotic convergence rate is observed for the case.
Note that the oscillations of $f_2$ are not harmonic, with the frequency increasing towards the right endpoint of the approximation interval. This seems to have the effect of reducing the advantage of schemes with optimized resolution properties.

\begin{figure}[t]
\begin{center}
  \subfigure[Approximation of $f_1$]{
    \includegraphics[width=5cm]{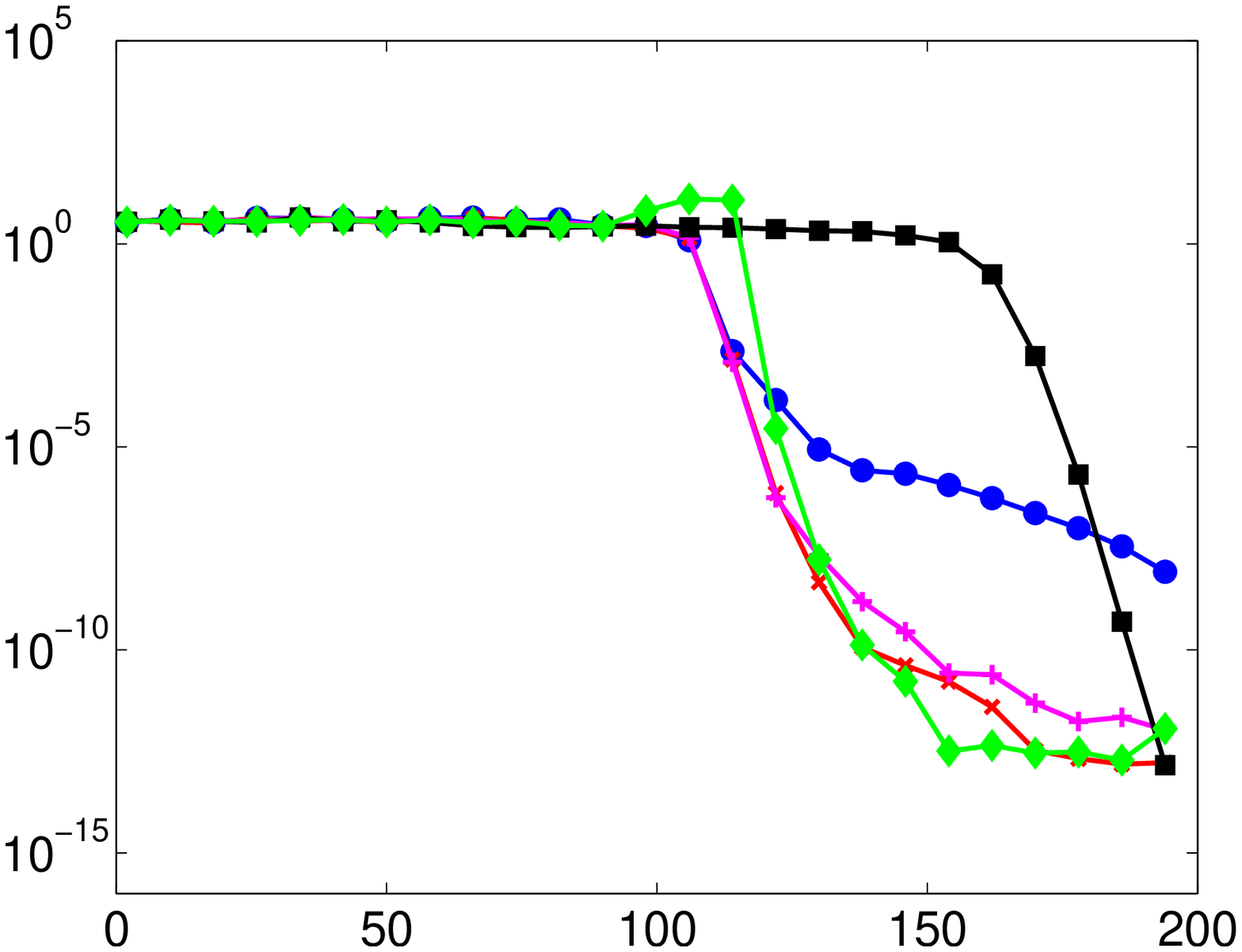}
  }
  \subfigure[Approximation of $f_2$]{
    \includegraphics[width=5cm]{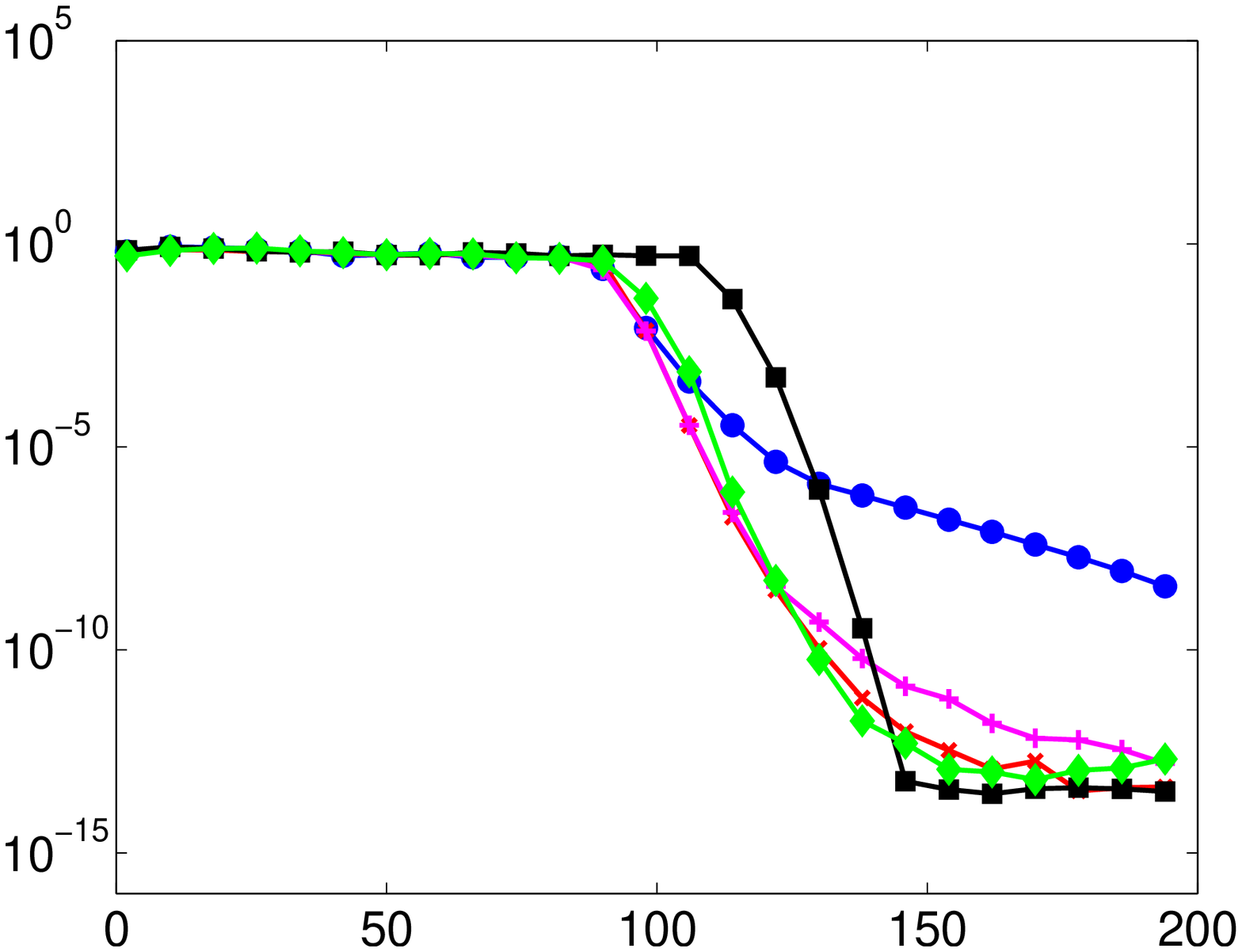}
  }
 \caption{The error $\| f - g_n \|_{L^{\infty}_{[-1,1]}}$, where $g_n$ is Chebyshev expansion (squares), or the Fourier extension approximation with $T=1+n^{-2/3}$  (circles), $T=1+n^{-1/2}$ (crosses), $T=\frac{8}{7}$ (diamonds) and a formally optimal $T$ given by \eqref{eq:optimal_T} with $\epsilon_{\mathrm{tol}}=$1e-13 (pluses).}\label{fig:num_comp}
\end{center}
\end{figure}

\section{Conclusions and challenges} \label{s:conclusions}
The purpose of this paper was to describe and analyze the observation that Fourier extensions are eminently well suited for resolving oscillations.  To do this, we derived an exact expression for the so-called resolution constant of the continuous/discrete Fourier extension, and offered an explanation as to why this constant need not be realized in the case of large $T$.

Although we have touched upon the subject of numerical behaviour of Fourier extensions in this paper, we have not provided a full description.  As mentioned, the continuous/discrete Fourier extensions both result in severe ill-conditioning.  However, numerical examples in this paper and elsewhere (see \cite{BoydFourCont,brunoFEP,LyonFESVD,LyonFast}) point towards an apparent contradiction.  Namely, despite this ill-conditioning, the best achievable accuracy can still be very high (close to machine epsilon for the discrete extension).  It transpires that this effect can be quite comprehensively explained, with the main conclusion being the following: the discrete Fourier extension, although the result of solving an ill-conditioned linear system, is in fact numerically stable.  A paper on this topic is in preparation \cite{BADHJMVFEStability}.  Incidentally, the analysis presented therein also allows one to make rigorous the arguments given in \S \ref{ss:resolution_numerical} on the differing resolution constants of 
theoretical and numerical Fourier extensions for large $T$.

The discrete Fourier extension introduced in this paper uses values of $f$ on a particular nonequispaced mesh to compute the Fourier extension.  As mentioned, one important use of Fourier extensions is to solve PDE's in complex geometries.  In this setting, such meshes are usually infeasible.  For this reason, it is of interest to consider the question of how to compute rapidly convergent, numerically stable Fourier extensions with good resolution power from function values on equispaced meshes  (as discussed, related methods based on the FE--Gram extension are developed in \cite{albin2011,bruno2010high,lyon2010high}).  It transpires that this can be done, and we will report the details in the upcoming paper  \cite{BADHJMVFEStability} (see also \cite{LyonFESVD}).

Finally, we remark that there has been little discussion in this paper on the computational cost of computing Fourier extensions.  Instead, we have focused on approximation-theoretic properties, such as resolution.  However, M. Lyon has recently introduced a fast algorithm for this purpose \cite{LyonFast}, which allows for computation of Fourier extensions in $\BIGO (n (\log n)^2)$ operations.

\section*{Acknowledgements}
The authors would like to thank Rodrigo Platte for pointing out the relation to the Kosloff--Tal--Ezer mapping.  They would also like to thank Jean--Fran\c{c}ois Maitre for useful discussions and comments.

\bibliographystyle{abbrv}
\bibliography{resolution}

\end{document}